\documentclass[final,leqno,onefignum,onetabnum]{siamltex1213}
\usepackage{amssymb}
\usepackage{amsmath}
\usepackage{graphicx}
\usepackage{subfigmat}
\usepackage{amsfonts}
\usepackage{latexsym} 
\usepackage{ wasysym}
\usepackage{newlfont}
\usepackage{leqno}
\usepackage[usenames]{xcolor}
\usepackage{stmaryrd}
\newtheorem{thm}{Theorem}[section]
\newtheorem{assume}{Assumption}

\newtheorem{rem}[thm]{Remark}
\numberwithin{equation}{section}
 
\makeatother 
\title{Analysis of Multipatch Discontinuous {G}alerkin IgA Approximations
    to  Elliptic Boundary Value Problems 
    \thanks{This work was supported by  Austrian Science Fund (FWF) under the grant NFN S117-03.}} 
 
\author{ Ulrich Langer\footnotemark[1]\ \footnotemark[2]
\and Ioannis Toulopoulos\footnotemark[1]\ \footnotemark[3]
        }

\begin{document}
\maketitle
\slugger{SINUM}{xxxx}{xx}{x}{x--x}

\renewcommand{\thefootnote}{\fnsymbol{footnote}} 
 \footnotetext[1]{Johann Radon Institute for Computational and Applied Mathematics (RICAM)}
 \footnotetext[2]{ulrich.langer@ricam.oeaw.ac.at}  
 \footnotetext[3]{ioannis.toulopoulos@oeaw.ac.at}  
\renewcommand{\thefootnote}{\arabic{footnote}}

\begin{abstract}
         In this work, we study the approximation properties of 
         multi-patch dG-IgA methods,
         that apply the  multipatch Isogeometric Analysis (IgA) discretization concept  and  
         the discontinuous Galerkin (dG) technique on the interfaces between the patches,
         for solving linear 
	 diffusion problems with
         diffusion coefficients that may be discontinuous across the patch interfaces.
         The computational domain is divided 
	 into non-overlapping
	 sub-domains, called patches in IgA, 
         where 
         $B$-splines, or NURBS
         finite dimensional approximations  spaces are constructed.
         The solution of the problem is approximated in every sub-domain without imposing any 
         matching grid
         conditions and without any continuity requirements for the discrete solution 
	 across 	 the interfaces.
         Numerical fluxes with interior penalty  jump terms are applied
         in order to treat the discontinuities of the discrete solution on the 
         interfaces.
         We provide a rigorous a priori discretization error analysis          
         for problems set in 2d- and 3d- dimensional
         domains, with solutions belonging to  
         $W^{l,p}, l\geq 2,{\ } p\in ({2d}/{(d+2(l-1))},2]$.        
          In any case, we show optimal convergence rates of the discretization with
         respect to the dG - norm.
\end{abstract}
\begin{keywords}
           linear elliptic problems, discontinuous coefficients, 
	  discontinuous Galerkin discretization,
	  Isogeometric Analysis, 
          non-matching grids, low regularity solutions, a priori discretization error estimates
\end{keywords}
\begin{AMS}
65N12, 65N15, 65N35
\end{AMS}

\pagestyle{myheadings}
\thispagestyle{plain}
\markboth{U. Langer, I. Toulopoulos}{dG-IgA for Elliptic problems}

\section{Introduction}
The finite element methods (FEM) and, in particular, discontinuous Ga\-lerkin (dG) 
finite element methods are very often used for solving elliptic boundary value problems which arise
from engineering applications, see, e.g., \cite{Hughes_FEbook},\cite{BQ_Li_DGbook}. 
Although the isoparametric FEM and even FEM with curved finite elements 
have been proposed and analyzed  long time ago,
cf. 
\cite{Zlamal:1973a},
\cite{Ciarlet_FEbook}, 
\cite{Hughes_FEbook},
the quality of the numerical results
for realistic problems in complicated geometries
depends  on the quality of the discretized geometry
(triangulation of the domain), which is usually  performed by a mesh generator. 
In many practical situations, 
extremely fine meshes are required 
around fine-scale geometrical objects, 
singular corner points  
etc.
in order to achieve   numerical solutions with desired resolution. 
This fact leads to an increased number of degrees of freedom, 
and thus to an increased overall computational cost for solving 
the discrete problem, 
see, e.g., \cite{Turek_INS_book} 
for fluid dynamics applications. 
\par
Recently,
the  Isogeometric Analysis (IgA) concept 
has been applied for approximating solutions of elliptic
problems \cite{HUGHE_IGA_CMAME_2005}, \cite{Bazilevs_IGA_ERR_ESti2006}.
IgA generalizes and improves the classical FE (even isoparametric FE) methodology in the following 
direction: 
complex technical computational domains can be exactly represented as images of 
some parameter domain, where the mappings  
are constructed by using superior classes of basis functions 
like $B$-Splines,  or Non-Uniform Rational ${B}$-Splines (NURBS),
see, e.g.,  \cite{Shumaker_Bspline_book} and \cite{Piegl_NURBS_book}. 
The same 
class of functions is used to approximate the exact solution without increasing
the computational cost for the computation of the resulting stiffness matrices 
\cite{Hughes_IGAbook},
systematic $hpk$ refinement procedures can easily be  developed 
\cite{Beirao_Buffa_NumMathIGAErr2011}, 
and, last but not least,
the method can be materialized in parallel environment incorporating 
fast domain decomposition solvers
\cite{KleissPechsteinJuttlerTomar:2012a},
\cite{BeiraoChoPavarinoScacchi:2013a},
\cite{ApostolatosSchmidtWuencherBletzinger:2014a}.

\par
 During the last two decades, there has been an increasing  interest in 
discontinuous Galerkin finite element methods  for the numerical solution of several types of
partial differential equations, see, e.g., 
\cite{BQ_Li_DGbook}.
This is due to the advantages of the local approximation
spaces without continuity requirements that dG methods offer, see, e.g.,
\cite{CockMurEllipDG}, \cite{ERN_DGbook}, \cite{Rivierebook} and \cite{Maximiliam_DG_DD}.
\par
In this paper, we combine the best features of the two aforementioned methods,
and develop a powerful discretization method that we call 
multipatch discontinuous Galerkin Isogeometric Analysis (dG-IgA).
In particular, 
we study and analyze the dG-IgA approximation properties to elliptic 
boundary value problems with discontinuous coefficients. 
It well known that the solutions of this type of problems  
are in general not enough smooth, see, e.g. 
\cite{Kellog_DiscDifCoef_1975},
\cite{Knees:2004a}, and 
the approximate method can not produce an (optimal) accurate solution.
The problem is set in a complex, bounded Lipschitz
domain $\Omega \subset \mathbb{R}^d, d=2,3$, which is  subdivided
in a union of non-overlapping sub-domains, say $\cal{S}(\Omega):=\{\Omega_i\}_{i=1}^N$. 
Let us assume that the 
discontinuity of the diffusion coefficients is only observed  
across sub-domain boundaries (interfaces).
The weak solution of the problem is approximated in every sub-domain applying IgA
methodology, \cite{Bazilevs_IGA_ERR_ESti2006}, without matching grid conditions
along the interfaces, as well  without imposing continuity requirements 
for the approximation spaces across the interfaces. By construction,
dG methods use discontinuous approximation spaces utilizing numerical 
fluxes on the interfaces, 
\cite{KirbyKarniadakis2005},  
and  have been efficiently  used 
 for solving problems on non-matching grids in the past, 
\cite{Maximiliam_DG_DD}, \cite{Maximiliam_DG_BDDC},
 \cite{Karakashian_SCHWARTZ_DG_Elliptic}.
Here, emulating  the dG finite element methods,   the numerical scheme is formulated  by applying numerical fluxes
with interior penalty coefficients on the interfaces of the sub-domains 
(patches), and using IgA formulations in every patch independently. 
A crucial point
in the presented work, is the expression of the  numerical flux interface terms as a sum over 
the  micro-elements edges taking note of the non-matching sub-domain grids. 
This gives 
the opportunity to proceed in the error analysis by applying
the trace  inequalities locally  as in dG finite element methods.
There are many papers, which present dG finite element approximations for elliptic problems,
see, e.g.,
\cite{CockMurEllipDG}, \cite{RiviereWheelerGirault2001},    
the monographs \cite{Rivierebook},\cite{ERN_DGbook}, 
and, in particular, for the discontinuous coefficient case, 
\cite{Maximiliam_DG_DD}, \cite{Ern_DG_Hetrg_Diff}. 
However, there are only a few publications 
on the dG-IgA and their analysis.
In \cite{Brunero:2012a}, the author presented 
discretization error estimates for the dG-IgA 
of plane (2d) diffusion problems 
on meshes matching across the patch boundaries 
and under the assumption of sufficiently smooth solutions. 
This analysis obviously carries over to 
plane linear elasticity problems which have recently been studied 
numerically in \cite{ApostolatosSchmidtWuencherBletzinger:2014a}.
In \cite{EvansHughes:2013a}, the dG technology 
has been used to handle no-slip boundary conditions and multi-patch
geometries for IgA of Darcy-Stokes-Brinkman equations.
DG-IgA discretizations of heterogenous diffusion problems on 
open and closed surfaces, which are given by a multipatch NURBS representation, 
are constructed and rigorously analyzed in \cite{LangerMoore:2014a}.

\par
In the first part of this paper, 
we give a priori error estimates in the $\|.\|_{dG}$ norm 
    under the usual    regularity assumption imposed on the exact solution, 
    i.e. $u\in W^{1,2}(\Omega)\cap W^{l\geq 2,2}(\cal{S}(\Omega))$. 
    Next, we consider the model problem with low regularity solution 
$u\in W^{1,2}(\Omega)\cap W^{l,p}(\cal{S}(\Omega))$, 
with $l\geq 2$ and $p \in  (\frac{2d}{d+2(l-1)},2)$,
and
    derive error estimates in the $\|.\|_{dG}$. These estimates are optimal with respect to the space size discretization.
    We note that the error analysis in the case of  low regularity solutions includes many ingredients
    of the dG FE error analysis of 
    \cite{Riviere_DG_lowReg} and \cite{Ern_DG_Hetrg_Diff}.
    To the best of our knowledge, optimal error analysis for 
    IgA discretizations combined with dG techniques for solving
    elliptic problems with discontinuous coefficients 
    in general domains $\Omega \subset \mathbb{R}^d$, $d=2,3$,
    have not been yet  presented in the literature.
\par
  The paper is organized as follows. In Section 2, 
  our model diffusion problem is described. 
  Section 3 introduces some notations. 
  The local $\mathbb{B}_h(\cal{S}(\Omega))$ approximation space and the numerical scheme are
  also presented. Several auxiliary results and the analysis of the method for the
  case of usual regularity solutions are provided  in Section 4.  
  Section 5 is devoted to the analysis of the method for low regularity solutions. 
  Section 6 includes several numerical examples that verify the theoretical convergence rates. 
  Finally, we draw some conclusions.
  
\noindent

\section{The model problem}
Let $\Omega$ be a bounded Lipschitz domain in $\mathbb{R}^d,{\ }d=2,3$, with 
the boundary $\partial \Omega$.
For simplicity, we restrict our study to the  model problem
\begin{align}
 \label{0}
    -\mathrm{div}(\alpha\nabla u) = {f}{\ }  \text{in}{\ } \Omega, {\ }
    \text{and}{\ }{u}     = {u}_D {\ }\text{on $\partial \Omega $}, 
\end{align}
where $f$ and $ u_D$ are given smooth data.  
In (\ref{0}), 
$ \alpha $ is the diffusion coefficient and assume be bounded by above and below by 
strictly positive constants. 
  
 \par
The weak formulation is to find  a function $u\in W^{1,2}(\Omega)$ such that $u:=u_D$ 
on $\partial \Omega$ and  satisfies
\begin{subequations}\label{4}
\begin{alignat}{2}\label{4a}
 a(u,\phi)=&l(\phi),{\ }  \forall  \phi\in W^{1,2}_{0}(\Omega),\\
\intertext{where}
\label{4b}
a(u,\phi)= & \int_{\Omega}\alpha\nabla u\nabla\phi\,dx, \quad\text{and}\quad  l(\phi)=\int_{\Omega}f\phi\,dx.
\end{alignat}
\end{subequations}
Results concerning the existence and uniqueness of the solution $u$ of problem (\ref{4}) can
                 be derived by a simple application of Lax-Milgram Lemma, 
                 \cite{Evans_PDEbook}.  
To avoid unnecessary long formulas below,  we only considered in (\ref{0})  non-homogeneous
                  Dirichlet boundary conditions on $\partial \Omega$.
                  However, the analysis can be easily generalized to 
                  Neumann and Robin  type boundary conditions 
                 on a part of $\partial \Omega$, since they are naturally introduced
                 in the dG  formulation.
\section{Preliminaries - dG notation}
Throughout this work, we denote by  $L^p(\Omega), p>1$ \textit{the Lebesgue spaces} for which 
$\int_{\Omega}|u(x)|^p\,dx < \infty$, endowed with the norm
$\|u\|_{L^p(\Omega)} = \big(\int_{\Omega}|u(x)|^p\,dx\big)^{\frac{1}{p}}$.
By $\mathcal{D}(\Omega)$, we define the the space of $C^{\infty}$ functions
with compact support in $\Omega$, and by $C^{k}(\Omega)$ the set of functions with $k-th$ order
continues derivatives. In dealing with differential operators in Sobolev spaces, 
we use the following common conventions.
For any \textit(multi-index) $\alpha=(\alpha_1,...,\alpha_d), {\ }\alpha_j\geq 0, j=1,...,d$, 
with degree $|\alpha| = \sum_{j=1}^d\alpha_j$, we define the \textit{differential operator}
\begin{align}\label{01a}
 D^a=D_1^{\alpha_1}\cdot\cdot\cdot D_d^{\alpha_d},\text{with}
 {\ }D_j=\frac{\partial}{\partial x_j}, D^{(0,...,0)}u=u.
\end{align}
We also denote by $W^{l,p}(\Omega)$, $l$ positive integer and $1\leq p \leq \infty$,
the Sobolev space functions  endowed with the norm
\begin{subequations}
\begin{align}\label{01b}
 \|u\|_{W^{l,p}(\Omega)} = \big(\sum_{0\leq |\alpha| \leq m} \|D^{\alpha}u\|_{L^p(\Omega)}^p\big)^{\frac{1}{p}},\\
 \|u\|_{W^{l,\infty}(\Omega)} = max_{0\leq |\alpha| \leq m} \|D^{\alpha}u\|_{\infty}.
\end{align}
\end{subequations}
For more details for the above definitions, we refer \cite{Adams_Sobolevbook}.
In the sequel we write $a\sim b$ if $c a \leq b \leq C a$, where $c,C$ are positive constants
indpented of the mesh size. 
\par
In order to apply the IgA methodology 
for the problem (\ref{0}), the domain
 $\Omega$ is subdivided into a union  of sub-domains 
 $\cal{S}(\Omega):=\{\Omega_i\}_{i=1}^{N}$,  such that

\begin{align}
\label{1}
\bar{\Omega} &= \bigcup _{i=1}^N \bar{\Omega}_i,\quad \text{with}\quad
\Omega_i\cap \Omega_j = \O{}, {\ } \text{if}{\ }j\neq i.
\end{align}
The subdivision of $\Omega$ assumed to be compatible with the discontinuities of $\alpha$,
\cite{Maximiliam_DG_DD}, \cite{Ern_DG_Hetrg_Diff}. 
In other words,
the diffusion coefficient assumed to be  constant in the interior of $\Omega_i$ and
 its discontinuities can appear only  
 on the interfaces $F_{ij}=\partial \Omega_i \bigcap \partial \Omega_j$.
 \par
As it is common in the IgA   analysis, 
we assume  a parametric domain $\widehat{D}$ of unit length, (e.g. $\widehat{D} = [0,1]^d$).
For any $\Omega_i$,  we associate  $n=1,...,d$  knot vectors $\Xi^{(i)}_n$ on  $\widehat{D}$,
which 
 create a mesh $T^{(i)}_{h_i,\widehat{D}} =\{\hat{E}_{m}\}_{m=1}^{M_i}$, where $\hat E_{m}$ are the micro-elements,
 see details in \cite{Hughes_IGAbook}.
 We shall refer $T^{(i)}_{h_i,\widehat{D}}$ as the \textit{parametric  mesh of $\Omega_i$}. 
 For every $\hat{E}_m\in T^{(i)}_{h_i,\widehat{D}}$ we denote by $h_{\hat{E}_m}$ \textit{its diameter } 
 and by $h_i=\max\{h_{\hat{E}_m}\}$ \textit{the  meshsize}  of $T^{(i)}_{h_i,\widehat{D}}$. 
We assume the following properties for every $T^{(i)}_{h_i,\widehat{D}}$,
 \begin{itemize}
  \item quasi-uniformity: for every  $\hat{E}_m\in T^{(i)}_{h_i,\widehat{D}}$ holds  ${h}_i\sim h_{\hat{E}_m}$,
   \item  for the  micro-element edges $e_{\hat{E}_m} \subset \partial \hat{E}_m$ 
             holds  $h_{\hat{E}_m} \sim e_{\hat{E}_m}$.
  \end{itemize}
  
On every $T^{(i)}_{h_i,\widehat{D}}$, we construct the finite dimensional 
space $\hat{\mathbb{B}}^{(i)}_{h_i}$ spanned by  
 $\mathbb{B}$-Spline basis functions of degree $k$, \cite{Hughes_IGAbook}, \cite{Shumaker_Bspline_book},
\begin{subequations}\label{3.2}
\begin{align}\label{3.2a}
 \hat{\mathbb{B}}^{(i)}_{h_i}=
span\{\hat{B}_j^{(i)}(\hat{x})\}_{j=0}^{dim(\hat{\mathbb{B}}_{h_{i}}^{(i)})}, 
\intertext{where every $\hat B_j^{(i)}(\hat{x})$ base function  in (\ref{3.2a}) 
      is derived by means of tensor
           products of  one-dimensional $\mathbb{B}$-Spline basis functions, e.g.}
 \label{3.2b}
\hat B_j^{(i)}(\hat{x}) = \hat B_{j_1}^{(i)}(\hat{x}_1)\cdot\cdot\cdot \hat B_{j_d}^{(i)}(\hat{x}_d).
\end{align}
\end{subequations}
For simplicity, we assume that  the basis functions
of every $\hat{\mathbb{B}}_{h_{i}}^{(i)}, i=1,...,N$ are of the same degree $k$.
We denote by $\tilde{D}^{(i)}_{\hat{E}}$ the support extension of $\hat{E}\in T^{(i)}_{h_i,\widehat{D}}$. 
 \par
 Every sub-domain  $\Omega_i\in \cal{S}(\Omega), i=1,...,N$, is exactly represented through a parametrization 
 (one-to-one mapping), \cite{Hughes_IGAbook}, having the form
\begin{subequations}\label{2}
\begin{align}
\label{2_0}
 \mathbf{\Phi}_i: \widehat{D} \rightarrow \Omega_i, &\quad
 \mathbf{\Phi}_i(\hat{x}) = \sum_{j}C^{(i)}_j \hat{B}_j^{(i)}(\hat{x}):=x\in \Omega_i,\\
 \label{2_1}
 \text{with}&\quad \hat{x} = \mathbf{\Psi}_i(x):=\mathbf{\Phi}^{-1}_i(x),
 \end{align}
\end{subequations}
   where $C_j^{(i)}$ are the control points.
 Using  $\mathbf{\Phi}_i$, 
we  construct a mesh $T^{(i)}_{h_i,\Omega_i} =\{E_{m}\}_{m=1}^{M_i}$ for every  $\Omega_i$,
whose vertices are the images of the vertices  of
the corresponding mesh $T^{(i)}_{h_i,\widehat{D}}$ through  $\mathbf{\Phi}_i$. 
 If $h_{\Omega_i}=\max\{h_{E_m}\}, {\ }E_m\in T^{(i)}_{h_i,\Omega_i}$ is the 
 {sub-domain $\Omega_i$ mesh size}, then  based on  definition (\ref{2}) 
 of $\mathbf{\Phi}_i$, 
 there is a constant $C:=C(\|\mathbf{\Phi}_i\|_{\infty})$ such that
 $h_i \sim C h_{\Omega_i}.$
In what follows, we denote  the sub-domain mesh size by $h_i$ without the constant 
$C:=C(\|\mathbf{\Phi}_i\|_{\infty})$ explicitly appearing. 
\par
The mesh of $\Omega$ is considered to be $T_h(\Omega)=\bigcup_{i=1}^N T^{(i)}_{h_i,\Omega_i}$,
where we note that there are no matching mesh requirements on the interior interfaces 
$F_{ij}=\partial \Omega_i \bigcap \partial \Omega_j, i\neq j$.
For the sake of brevity in our notations,  the interior  faces of the boundary of the  
sub-domains are denoted by $\mathcal{F}_{I}$ and 
 the collection of the  faces  that belong to  
           $\partial \Omega$  by $\mathcal{F}_B$, 
           e.g. $F\in \mathcal{F}_B$
           if there is a $\Omega_i$ such that $F=\partial \Omega_i \bigcap \partial \Omega$.
           We denote the set of all sub-domain faces by $\mathcal{F}.$ 
 \par      
 Lastly, 
 we  define on $\Omega$ the finite dimensional $\mathbb{B}-$Spline space\\
$\mathbb{B}_h(\cal{S}(\Omega))=\mathbb{B}^{(1)}_{h_{1}}\times ... \times \mathbb{B}^{(N)}_{h_{N}}$,
where every  $\mathbb{B}^{(i)}_{h_{i}}$ is defined 
on $T^{(i)}_{h_i,\Omega_i}$ as follows
\begin{align}\label{3}
 \mathbb{B}^{(i)}_{h_{i}}:=\{B_{h_{i}}^{(i)}|_{\Omega_i}: B_h^{(i)}(\hat{x})=
  \hat{B}_h^{(i)}\circ \mathbf{\Psi}_i({x}),{\ }
 \forall \hat{B}_h^{(i)}\in \hat{\mathbb{B}}^{(i)}_{h_i} \}. 
\end{align}
        We define the union support in physical
       sub-domain $\Omega_i$  as  $D^{(i)}_{E}:=\mathbf{\Phi}(\tilde{D}^{(i)}_{\hat{E}})$.
  Since $ \mathbf{\Phi}_i(\hat{x})\in \hat{\mathbb{B}}_h^{(i)}$, 
the components  $\Phi_{1,i},...,\Phi_{d,i} \in \hat{\mathbb{B}}_h^{(i)}$  
are smooth functions and hence  
 there exist constants $c_m,c_M$ such that
\begin{align}\label{2_a}
 c_m \leq |det(\mathbf{\Phi}^{'}_{i}(\hat{x}))| \leq c_M, {\ }i=1,...,d, {\ }\text{for all} {\ }\hat{x} \in \hat{D}
\end{align}
where $\mathbf{\Phi}^{'}_{i}(\hat{x})$ denotes the 
Jacobian matrix $\frac{\partial({x}_1,...,{x}_d)}{\partial(\hat{x}_1,...,\hat{x}_d)}$.
\par
Now, for any $\hat{u}\in W^{m,p}(\hat{D}), m\geq 0, p>1$,  we define the function  
\begin{align}\label{2_b}
 \mathcal{U}(x)=\hat{u}(\mathbf{\Psi}_i(x)),{\ }x\in \Omega_i,
\end{align}
where   $\mathbf{\Psi}$ is defined in   (\ref{2_1}).
For the error analysis presented below, it is 
 necessary to show the   relation
\begin{align}\label{2_c}
 C_m  \|\hat{u}\|_{W^{m,p}(\hat{D})} \leq  \|\mathcal{U}\|_{W^{m,p}(\Omega_i)} \leq C_M \|\hat{u}\|_{W^{m,p}(\hat{D})},
\end{align}
where  the constants $C_m, C_M$ depending on 
$$
C_m:=C_m(\max_{m_0\leq m}(\|D^{m_0}\mathbf{\Phi}_i\|_{\infty}),\|det(\mathbf{\Psi}^{'}_{i})\|_{\infty})
$$ 
and
$$
C_M:=C_M(\max_{m_0\leq m}(\|D^{m_0}\mathbf{\Psi}_i\|_{\infty}),\|det(\mathbf{\Phi}^{'}_{i})\|_{\infty})
$$ correspondingly.

Indeed, for any  $\hat{u}\in W^{m,p}(\hat{D})$ we can find a sequence 
$\{\hat{u}_j\}\subset C^{\infty}(\bar{\hat{D}})$ converging to $\hat{u}$
in $\|.\|_{W^{m,p}(\hat{D})}$, by the chain rule in (\ref{2_b})  we obtain
\begin{align}\label{2_d}
D_{{x}}(\mathbf{\Psi}_i({x}))^{-1} D\mathcal{U}_j(x) =& D\hat{u}_j(\mathbf{\Psi}_i(x)).%
\intertext{Then for any multi-index $m$  we can get  the following formula}
 \label{2_e}
 D^m\mathcal{U}_j(x) =& \sum_{m_0\leq m} P_{m,m_0}(x)D^{m_0}\mathcal{U}_j(x),
\end{align}
where $P_{m,m_0}(x)\in \mathbb{B}_h^{(i)}$ is a polynomial of degree less than $k$ and
includes the various derivatives of $\mathbf{\Psi}_i(x)$. 
Multiplying (\ref{2_e}) by  $\varphi(x) \in \cal{D}(\Omega_i)$, and integrating by parts
we have
\begin{align}\label{2_f}
 (-1)^{|m|}\int_{\Omega_i}\mathcal{U}_j(x)D^{m}\varphi(x)\,dx=
 \sum_{m_0 \leq m}\int_{\Omega_i}P_{m,m_0}(x)D^{m_0}\mathcal{U}_j(x)\varphi(x)\,dx.
\end{align}
We transfer the integral in (\ref{2_f}) to integrals over $\hat{D}$ and use 
 the change of variable $x=\mathbf{\Phi}_i(\hat{x})$ to obtain
\begin{multline}\label{2_i}
 (-1)^{|m|}\int_{\hat{D}}\hat{u}_j(\hat{x})D^{m}\varphi(\mathbf{\Phi}_i(\hat{x}))|det(\mathbf{\Phi}^{'}_{i}(\hat{x}))|\,d\hat{x}= \\
 \sum_{m_0 \leq m}\int_{\hat{D}}P_{m,m_0}(\mathbf{\Phi}_i(\hat{x}))D^{m_0}\hat{u}_j(\hat{x})
 \varphi(\mathbf{\Phi}_i(\hat{x}))|det(\mathbf{\Phi}^{'}_{i}(\hat{x}))|\,d\hat{x}.
\end{multline}
But it holds that $D^{m_0}\hat{u}_j\rightarrow D^{m_0}\hat{u}$ in $\|.\|_{L^p(\hat{D})}$, thus taking the limit $j\rightarrow \infty$ 
in (\ref{2_i}) and transferring  the integrals back to $\Omega_i$, we can derive (\ref{2_f}) 
with respect to $\mathcal{U}$.
We conclude that (\ref{2_e}) holds in the distributional sense,  and therefore
\begin{multline}\label{2_j}
 \int_{\Omega_i}|D^m\mathcal{U}(x)|^p\,dx \leq C_p \int_{\Omega_i}\sum_{m_0\leq m}\big|P_{m,m_0}(x)D^{m_0}\mathcal{U}(x)|^p\,dx \leq \\
 \hspace*{-10mm}{
 C_p \max_{m_0\leq m}\big(\max_{x\in \Omega_i}( P_{m,m_0}(x))\big)\sum_{m_0\leq m}\int_{\Omega_i}\big|D^{m_0}\mathcal{U}(x)|^p\,dx \leq
 }\\
 C_p \max_{m_0\leq m}\big(\max_{x\in \Omega_i}( P_{m,m_0}(x))\big)
 \max_{\hat{x}\in\hat{D}}(|det(\mathbf{\Phi}^{'}_{i}(\hat{x}))|\sum_{m_0\leq m}\int_{\hat{D}}\big|D^{m_0}\hat{u}(\hat{x})|^p\,d\hat{x} \leq \\
 \hspace*{-10mm}{
 C(\max_{m_0\leq m}(\|D^{m_0}\mathbf{\Psi}_i(x)\|_{\infty},\|det(\mathbf{\Phi}^{'}_{i}(\hat{x}))\|_{\infty})
 \sum_{m_0\leq m}\big|D^{m_0}\hat{u}(\hat{x})|^p_{W^{m_0,p}(\hat{D})}.}
\end{multline}
This proves the  ``right inequality'' of (\ref{2_c}).  
The  ``left inequality'' of (\ref{2_c}) can be derived following
the same arguments as above   
  using the  change of variable $\hat{x}=\mathbf{\Psi}_i(x)$.   
\subsection{The numerical scheme}
We use the $\mathbb{B}-$Spline spaces $\mathbb{B}^{(i)}_h$ defined in (\ref{3}) for approximating  the solution of (\ref{4})
in every sub-domain $\Omega_i$. 
Continuity requirements for $\mathbb{B}_h(\cal{S}(\Omega))$  are not imposed on the interfaces 
$F_{ij}$ of the sub-domains, clearly $\mathbb{B}_h(\cal{S}(\Omega))\subset L^2(\Omega)$ but 
$\mathbb{B}_h(\cal{S}(\Omega))\nsubseteq W^{1,2}(\Omega)$.
Thus, the problem (\ref{4}) is discretized by discontinuous Galerkin techniques on
$F_{ij}$, \cite{Maximiliam_DG_DD}.
Using the notation $\phi_h^{(i)}:=\phi_h|_{\Omega_i}$,
we define the average and the jump of $\phi_h$ on $F_{ij}\in\mathcal{F}_I$ respectively by
\begin{subequations}\label{5a}
\begin{align}
   \{\phi_h\}:=\frac{1}{2}(\phi_h^{(i)}+\phi_h^{(j)}), & {\ }\llbracket \phi_h \rrbracket :=\phi_h^{(i)} - \phi_h^{(j)},\\
   \intertext{ and for $F_i\in \mathcal{F}_B$}
   \label{3.16c}
  \{\phi_h\}:=\phi_h^{(i)}, &{\ }\llbracket \phi_h\rrbracket :=
  \phi_h^{(i)}.
\end{align}
\end{subequations}

The dG-IgA method reads  as follows: find $u_h\in \mathbb{B}_h(\cal{S}(\Omega))$ such that
\begin{subequations}\label{8}
\begin{flalign}\label{8a}
 a_h(u_h,\phi_h)=&l(\phi_h)+p_D(u_D,\phi_h),{\ } \forall \phi_h \in \mathbb{B}_h(\cal{S}(\Omega)),
 \intertext{where}
 \label{8b}
 a_h(u_h,\phi_h) = & \sum_{i=1}^N a_i(u_h,\phi_h)-\sum_{F_{ij}\in\mathcal{F}}\frac{1}{2}s_i(u_h,\phi_h)+p_i(u_h,\phi_h),
 \intertext{ with  the bi-linear forms} 
 \label{8c}
 a_i(u_h,\phi_h) =& \int_{\Omega_i}\alpha\nabla u_h\nabla\phi_h\,dx, \\
 \label{8d}
 s_i( u_h,\phi_h)=& \int_{F_{ij}\in \mathcal{F}} \{\alpha\nabla u_h\}\cdot \mathbf{n}_{F_{ij}} \llbracket \phi_h \rrbracket \,ds, \\
 \label{8e}
 p_i(u_h,\phi_h) =&
                   \begin{cases}
 	               p_{i_I}(u_h,\phi_h)  & = \int_{F_{ij}\in \mathcal{F}_I} \Big(\frac{\mu \alpha^{(j)}}{h_j}+\frac{\mu \alpha^{(i)}}{h_i}\Big) \llbracket  u_h \rrbracket \llbracket \phi_h \rrbracket\,ds, \\
 	                p_{i_B}(u_h,\phi_h) & = \int_{F_{i}\in \mathcal{F}_B} \frac{\mu \alpha^{(i)}}{h_i} \llbracket  u_h \rrbracket \llbracket \phi_h \rrbracket\,ds,\\
                   \end{cases}\\
 \label{8f}
 p_D(u_D,\phi_h)=&  \int_{F_{i}\in \mathcal{F}_B} \frac{\mu \alpha^{(i)}}{h_i}  u_D  \phi_h\,ds,
\end{flalign}
\end{subequations}
where the unit normal vector  $\mathbf{n}_{F_{ij}}$ is oriented from $\Omega_i$ towards the interior of $\Omega_j$
and the  parameter $\mu>0$
will be specified later in the error analysis, cf. \cite{Maximiliam_DG_DD}.
\par
For notation convenience in what follows, we will use the same  expression 
$$
 \int_{F_{ij}\in \mathcal{F}} \big(\frac{\mu \alpha^{(j)}}{h_j}+\frac{\mu \alpha^{(i)}}{h_i}\big) \llbracket  u_h \rrbracket \llbracket \phi_h \rrbracket\,ds,
$$
for both cases, $F_{ij}\in \mathcal{F}_I$ and $F_{i}\in \mathcal{F}_B$. In the later case we
will assume that $\alpha^{(j)}=0$.
\par
%
\begin{rem}
\label{remark_3.1}
We mention that, in \cite{Maximiliam_DG_DD}, Symmetric Interior Penalty (SIP) dG formulations 
have been considered by introducing harmonic averages of the diffusion coefficients on the interface 
symmetric fluxes. Furthermore, 
harmonic averages of the two different grid sizes have been used to penalize the jumps. 
The possibility of using other averages for constructing the diffusion terms in front of 
the consistency and penalty terms has been analyzed in many other works as well, see, e.g. 
\cite{Ern_DG_Hetrg_Diff} and \cite{HEINRICH_2003_IMA}. 
For simplicity of the presentation, we provide a rigorous analysis of  
the Incomplete Interior Penalty (IIP) 
forms (\ref{8d})  and  (\ref{8e}).
However, our analysis can easily be carried over to SIP dG-IgA 
that is prefered in practice for symmetric and positive definite (spd)
variational problems due to the fact that the resulting 
systems of algebraic equations are spd and, therefore, can be solved 
by means of some preconditioned conjugate gradient 
method.
\end{rem}
%
\begin{figure}
 \begin{subfigmatrix}{1}
{\includegraphics[width=0.63\textwidth, angle=0]{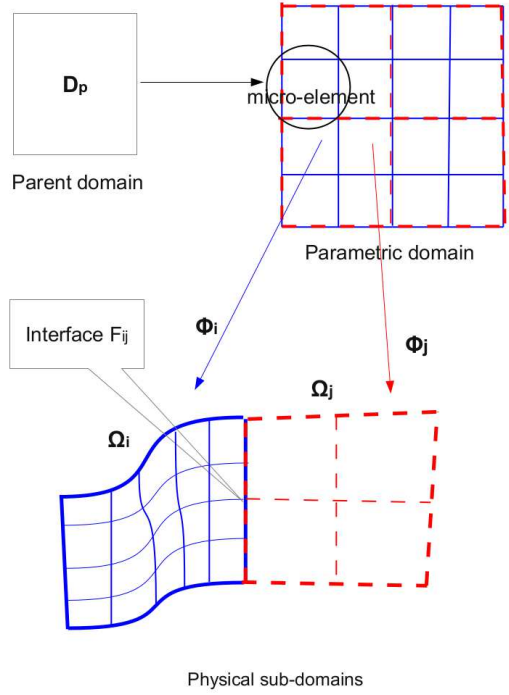}}
\end{subfigmatrix}
 \caption{The parent element,  the parametric domain and  two adjacent sub-domains.}
 \label{fg1_Domains}
\end{figure}
%
\section{Auxiliary results}
In order to proceed to error analysis, several auxiliary results must be shown for 
$u\in W^{l,p}(\cal{S}(\Omega))$ 
and $\phi_h\in \mathbb{B}_h(\cal{S}(\Omega))$.
 The general frame of  
  the proofs consists of  three steps: (i) the required  relations 
are expressed-proved  on a \textit{parent element} $D_p$, see Fig. \ref{fg1_Domains}, (ii)
the relations are ``transformed''  to $\hat{E}\in T^{(i)}_{h_i,\widehat{D}}$ 
using an affine-linear mapping and scaling arguments, (iii)
by virtue of the  mappings $\mathbf{\Phi}_i$ defined in (\ref{3}) and relations
(\ref{2_c}), we 
express 
the results in every $\Omega_i$.
\par
Let $D_p$ be  the  parent element e.g $[-x_b,x_b]^d \subset \mathbb{R}^d$,  with diameter $H_p$,
see Fig. \ref{fg1_Domains}.
$D_p$ is convex simply connected domain, thus for any $x\in \partial D_p, \exists x_0\in D_p$ such that
\begin{align}\label{9}
 (x-x_0)\cdot \mathbf{n}_{\partial D_p} \geq C_{H_p}.
\end{align}

\begin{lemma}\label{lemma1}
 For any $u\in W^{l,p}(D_p), l\geq 1, p>1$ there is a \\$C:=C_{H_p,d,p}$ such 
 that the following trace inequality holds true
 \begin{align}\label{trace}
  \int_{\partial D_p}|u(s)|^p \,ds \leq C \big(\int_{D_p}|\nabla u(x)|^p \,dx + \int_{D_p}|u(x)|^p \,dx\big).
 \end{align}
\end{lemma}
\begin{proof}
 For $r=(x-x_0)$ we have
\begin{align}\label{10}
 \int_{D_p}\nabla |u|^p\cdot r\,dx =
 \sum_{i=1}^d\int_{D_p}p|u|^{p-2}u\frac{\partial u}{\partial x_i}r_i\,dx=
  p\int_{D_p}|u|^{p-2}u\nabla u \cdot r \,dx.
\end{align}
The application of divergence theorem gives 

\begin{align}\label{11}
 \int_{D_p}\nabla |u|^p\cdot r \,dx=\int_{\partial D_p}|u|^p r\cdot \mathbf{n}_{\partial D_p}\,ds - \int_{D_p}|u|^p div(r)\,dx.
\end{align}

By (\ref{9}), (\ref{10}) and (\ref{11}) it follows that
\begin{multline*}
\int_{\partial D_p}|u|^p r\cdot \mathbf{n}_{\partial D_p}\,ds =p\int_{D_p}|u|^{p-2}u\nabla u \cdot r \,dx + 
\int_{D_p}|u|^p div(r)\,dx 
\end{multline*}
and by (\ref{9}), we get
\begin{multline}\label{12}
 C_{H_p}\int_{\partial D_p}|u|^p\,ds \leq p\int_{D_p}|u|^{p-2}u\nabla u \cdot r \,dx + \int_{D_p}|u|^pdiv(r)\,dx. \\
 \text{Applying  H\"older and Youngs inequalities, we have} \qquad \hskip 4cm\\
\int_{\partial D_p}|u|^p\,ds \leq C_{H_p}\Big( C_{1,p} \big( \int_{D_p}|u|^p\,dx +|\nabla u|^p\,dx\big) + C_d\int_{D_p}|u|^p\,dx \Big)\\
 \leq C_{H_p,d,p} \Big(\int_{D_p}|u|^p\,dx + \int_{D_p}|\nabla u|^p\,dx\Big) \\
 = C_{H_p,d,p}\Big( \|u\|^p_{L^{p}(D_p)}+ \|\nabla u\|^p_{L^{p}(D_p)}\Big).
\end{multline}

\end{proof}

We point out that similar proof has been given in 
\cite{Karakashian_SCHWARTZ_DG_Elliptic} in case of $p=2$.\\
$D_p$ can be considered as a reference element of any micro-element 
$\hat{E} \in T^{(i)}_{h_i,\widehat{D}}$ with the
  linear affine map 
\begin{align}\label{9a}
 \mathbf{\phi}_{\hat{E}}: D_p &\rightarrow \hat{E} \in T^{(i)}_{h_i,\widehat{D}},\quad
 \mathbf{\phi}_{\hat{E}}(x_{D_p}) = Bx_{D_p} + b,
\end{align}
where $|det(B)|=|\hat{E}|$, see \cite{BrenerScotBook}. 
By (\ref{9a}), we have that $|u|_{W^{l,p}(D_p)}=h^{l-\frac{d}{p}}_{\hat{E}}|\hat{u}|_{W^{l,p}(\hat{E})}$ and
then  we deduce by   (\ref{trace}) that
\begin{multline*}
 h^{-(d-1)}_{\hat{E}}\int_{e\in \partial \hat{E}}|u|^p\,ds \leq 
 C \Big(h_{\hat{E}}^{(0-\frac{d}{p})p}\int_{\hat{E}} |u|^p\,dx + h_{\hat{E}}^{p(1-\frac{d}{p})}\int_{\hat{E}}|\nabla u|^p\,dx \Big)
 \end{multline*}
 which  directly gives
 \begin{multline}\label{13}
   \int_{e\in \partial \hat{E}}|u|^p\,ds \leq C \Big( \frac{1}{h_{i}}\int_{\hat{E}} |u|^p\,dx +
   h_{i}^{p-1}\int_{\hat{E}}|\nabla u|^p\,dx\Big), {\ }\forall \hat{E}\in T^{(i)}_{h_i,\widehat{D}}.                                                       
\end{multline}
Summing over all micro-elements $\hat{E}\in T^{(i)}_{h_i,\widehat{D}}$, we have
\begin{align}\label{13a}
   \int_{\hat{F_i}\in \partial \widehat{D}}|u|^p\,ds \leq C \Big(\frac{1}{h_{i}}\int_{\widehat{D}} |u|^p\,dx +
   h_{i}^{p-1}\int_{\widehat{D}}|\nabla u|^p\,dx\Big).                                                       
\end{align}
Finally, by making use of (\ref{2_c}), we get  the trace inequality expressed
on every sub-domain
\begin{align}\label{13b}
   \int_{{F_{ij}}\in \mathcal{F}}|u|^p\,ds \leq C \Big( \frac{1}{h_{i}}\int_{\Omega_i} |u|^p\,dx +
   h_{i}^{p-1}\int_{\Omega_i}|\nabla u|^p\,dx\Big),                                                       
\end{align}
where   the constant $C$  is determined according to the  $C_m,C_M$ in (\ref{2_c}). 
\begin{lemma}\label{lemma2}(Inverse estimates)
 For all $\phi_h\in \hat{\mathbb{B}}^{(i)}_{h_i}$ defined on  $ T^{(i)}_{h_i,\widehat{D}}$, there is a constant
 $C$ depended on mesh quasi-uniformity parameters of the mesh but not on $h_i$, such that
 \begin{align}\label{14}
  \|\nabla \phi_h \|^p_{L^p(\widehat{D})} \leq \frac{C}{h_i^p}\| \phi_h \|^p_{L^p(\widehat{D})}
 \end{align}
\end{lemma}
\begin{proof}
 The restriction of $\phi_h|_{\hat{E}}$ is a ${B}-$Spline polynomial of the same order. 
 Considering the same polynomial 
 space on the  $D_p$ and  by the equivalence  of the norms on $D_p$ we have, \cite{BrenerScotBook},
 \begin{align}\label{15}
  \|\nabla \phi_h\|^p_{L^{p}(D_p)} \leq C_{D_p} \|\phi_h\|^p_{L^p(D_p)}.
 \end{align}
Applying scaling arguments and the  mesh quasi-uniformity properties of $ T^{(i)}_{h_i,\widehat{D}}$,
the left and the right hand side of (\ref{15}) can be expressed on every
$\hat{E}\in T^{(i)}_{h_i,\widehat{D}}$ as
\begin{align}\label{16}
  h_{i}^{p-\frac{d}{p}p}\|\nabla\phi_h\|^p_{L^{p}(\hat{E})}  \leq C h_{i}^{-\frac{d}{p}p} \|\phi_h\|^p_{L^p(\hat{E})},
 \end{align}
summing over all in (\ref{16}) $\hat{E}\in T^{(i)}_{h_i,\widehat{D}}$, we can easily deduce (\ref{14}).\\

\end{proof}

\begin{lemma}\label{lemma3}(trace inequality for $\phi_h\in \hat{\mathbb{B}}^{(i)}_{h_i}$)
 For all $\phi_h\in \hat{\mathbb{B}}^{(i)}_{h_i}$ defined on  $ T^{(i)}_{h_i,\widehat{D}}$ and for all $\hat{F}_i\in \partial \widehat{D}$,
 there is a constant
 $C$ depended on mesh quasi-uniformity parameters of the mesh but not on $h_i$, such that
 \begin{align}\label{18}
  \|\phi_h \|^p_{L^p(\hat{F_i}\in \partial \widehat{D})} \leq  \frac{C}{h_i^p}\| \phi_h \|^p_{L^p(\widehat{D})}
 \end{align}
\end{lemma}
\begin{proof}
Applying the same  scaling arguments as before and using the 
local quasi-uniformity of $T^{(i)}_{h_i,\widehat{D}}$, that is 
 for every  $\hat{e} \in \partial \hat{E}$ holds  $|\hat{e}| \sim h_i$,
we can show the following \textit{local trace inequality} 

\begin{align}\label{19}
  \|\phi_h\|^p_{L^{p}(\hat{e}\in \partial \hat{E})}   \leq  C h_i^{-p} \|\phi_h\|^p_{L^p(\hat{E})} 
 \end{align}
summing over all  $\hat{E}\in T^{(i)}_{h_i,\widehat{D}}$ that 
have an edge on $\hat{F_i}$ we deduce (\ref{18}).\\

\end{proof}

Next a Lemma for the relation among  the $|\phi_h|_{W^{l,p}(\widehat{D})}$ 
and $|\phi_h|_{W^{m,p}(\widehat{D})}$.
\begin{lemma}\label{lemma4_a}
 Let $\phi_h\in \hat{\mathbb{B}}^{(i)}_{h_i}$  such that
 $\phi_h \in W^{l,p}(\hat{E})\cap W^{m,q}(\hat{E}), {\ } \hat{E}\in   T^{(i)}_{h_i,\widehat{D}}$,
 and  $0\leq m\leq l, {\ }1\leq p,q \leq \infty$.
 Then there is a constant
 $C:=C(l,p,m,q)$ depended on mesh quasi-uniformity parameters of the mesh but not on $h_i$, such that
 \begin{align}\label{4.13}
  | \phi_h |_{W^{l,p}(\widehat{E})} \leq  C  h^{m-l-\frac{d}{q}+\frac{d}{p}}_{i} |  \phi_h |_{W^{m,q}(\widehat{E})}.
 \end{align}
\end{lemma}
\begin{proof}
We mimic the analysis of Chp 4 in \cite{BrenerScotBook}. 
For any $\phi_h\in \hat{\mathbb{B}}^{(i)}_{h_i}|_{D_p}$, we have that
\begin{align}\label{4.14}
|\phi_h|_{W^{l,p}(D_p)} \leq  C  |  \phi_h |_{W^{m,q}(D_p)}, {\ }\phi_h\in \hat{\mathbb{B}}^{(i)}_{h_i}|_{D_p}. 
\end{align}
Using the scaling arguments as in proof of (\ref{13}), 
\begin{align*}
h^{l-\frac{d}{p}}_{\hat{E}}|\phi_h|_{W^{l,p}(\hat{E})} \leq  & C  h^{m-\frac{d}{q}}_{\hat{E}}|\phi_h |_{W^{m,q}(\hat{E})}
\end{align*}
which directly implies
\begin{align}\label{4.16}
 |\phi_h|_{W^{l,p}(\hat{E})} \leq  C & h^{m-l-\frac{d}{q}+
 \frac{d}{p}}_{i}|\phi_h |_{W^{m,q}(\hat{E})},
{\ }\phi_h\in \hat{\mathbb{B}}^{(i)}_{h_i}.
\end{align}
For the particular case of $m=l=0$ in (\ref{4.13}), we have  that
\begin{align}\label{4.17}
 \|\phi_h\|_{L^p(\hat{E})} \leq C h_{i}^{d(\frac{1}{p}-\frac{1}{q})}\|\phi_h\|_{L^q(\hat{E})}.
\end{align}

\end{proof}

\subsection{Analysis of the dG-IgA discretization}
Next,  we study the convergence estimates of the method (\ref{8}) under the following regularity assumption.
\begin{assume}\label{Assumption1}
We assume for   the solution $u$ that 
$u\in W^{l,2}_{\cal{S}}:=W^{1,2}(\Omega)\cap W^{l,2}(\cal{S}(\Omega)),{\ }l \geq 2$.
\end{assume}
\par
We 
consider the enlarged  space 
 $W_h^{l,2}:=W^{l,2}_{\cal{S}}+ \mathbb{B}_h(\cal{S}(\Omega))$,
 equipped with the  broken dG-norm 
\begin{equation}\label{20}
 \|u\|^2_{dG} = \sum_{i=1}^N\Big(\alpha^{(i)}\|\nabla u^{(i)}\|^2_{L^2(\Omega_i)} +
                p_i(u^{(i)},u^{(i)}) \Big), {\ }u\in  W_h^{l,2}.
\end{equation}
For the error analysis is necessary  to show 
 the continuity and coercivity properties of the bilinear form
$a_h(.,.)$ of (\ref{8}). 
Initially, we  give a bound for the consistency terms.
\begin{lemma}\label{lemma4}
 For $(u,\phi_h)\in W_h^{l,2}\times \mathbb{B}_h(\cal{S}(\Omega))$,
 there are $C_{1,\varepsilon},C_{2,\varepsilon}>0$ such that for every $F_{ij}\in \mathcal{F}_I$
 \begin{multline}\label{21}
 |s_i|=
\Big| \int_{F_{ij}} \{\alpha\nabla u\}\cdot \mathbf{n}_{F_{ij}} ( \phi_h^{(i)}-\phi_h^{(j)}) \,ds\Big| \leq \\
C_{1,\varepsilon}\Big( h_i\alpha^{(i)}\|\nabla u^{(i)}\|^2_{L^2(F_{ij})}+  h_j\alpha^{(j)}\|\nabla u^{(j)}\|^2_{L^2(F_{ij})} \Big) + \\
   \frac{1}{C_{2,\varepsilon}}\Big(\frac{\alpha^{(i)}}{h_i}+\frac{\alpha^{(j)}}{h_j}\Big)\|\phi_h^{(i)}-\phi_h^{(j)}\|^2_{L^2(F_{ij})}.
 \end{multline}
\end{lemma}
\begin{proof}
 Expanding the terms and applying Cauchy-Schwartz inequality  yields
 \begin{multline*}
  |s_{i}|\leq C \Big| \int_{F_{ij}} \{\alpha\nabla u\}\cdot \mathbf{n}_{F_{ij}} ( \phi_h^{(i)}-\phi^{(j)}_h) \,ds\Big| \leq \\
  C \Big(\alpha^{(i)}\|\nabla u^{(i)}\|_{L^2(F_{ij})}+  \alpha^{(j)}\|\nabla u^{(j)}\|_{L^2(F_{ij})} \Big)\| \phi_h^{(i)}-\phi_h^{(j)}\|_{L^2(F_{ij})}. \\
  \\
  \text{Applying Young's inequality:} \hskip 8cm\\
  \\
  \alpha^{(i)}\|\nabla u^{(i)}\|_{L^2(F_{ij})}\| \phi_h^{(i)}-\phi_h^{(j)}\|_{L^2(F_{ij})} \leq 
  C_{1,\varepsilon} h_i \alpha^{(i)}\|\nabla u^{(i)}\|_{L^2(F_{ij})}^{2} + \\
  \frac{\alpha^{(i)}}{C_{2,\varepsilon} h_i} \| \phi_h^{(i)}-\phi_h^{(j)}\|_{L^2(F_{ij})}^2 \\
  \\
  \text{we obtain } \hskip 8cm\qquad \\
  \\
|s_i| \leq   C_{1,\varepsilon} h_i \alpha^{(i)}\|\nabla u^{(i)}\|_{L^2(F_{ij})}^{2} + C_{1,\varepsilon} h_j \alpha^{(j)}\|\nabla u^{(j)}\|_{L^2(F_{ij})}^{2} + \\
     \frac{\alpha^{(i)}}{C_{2,\varepsilon} h_i} \| \phi_h^{(i)}-\phi_h^{(j)}\|_{L^2(F_{ij})}^2 +
     \frac{\alpha^{(j)}}{C_{2,\varepsilon} h_j} \| \phi_h^{(i)}-\phi_h^{(j)}\|_{L^2(F_{ij})}^2  = \\
     C_{1,\varepsilon}\Big( h_i\alpha^{(i)}\|\nabla u^{(i)}\|^2_{L^2(F_{ij})}+  h_j\alpha^{(j)}\|\nabla u^{(j)}\|^2_{L^2(F_{ij})} \Big) + \\
   \frac{1}{C_{2,\varepsilon}}\Big(\frac{\alpha^{(i)}}{h_i}+\frac{\alpha^{(j)}}{h_j}\Big)\|\phi_h^{(i)}-\phi_h^{(j)}\|^2_{L^2(F_{ij})}.
 \end{multline*}
\end{proof}

\begin{rem}
  In case where  $ F_{i}\in \mathcal{F}_B$, the corresponding
  bound can be derived by setting in (\ref{21}) $\alpha^{(j)}=0$ and $ \phi_h^{(j)}=0$.
 \end{rem}
\begin{lemma}\label{lemma5}(Discrete Coercivity)
  Suppose $u_h\in \mathbb{B}_h(\cal{S}(\Omega))$ is the dG-IgA solution derived by (\ref{8}). 
 There exist a $C>0$ independent of $\alpha$ and  $h_i$, such that
 \begin{align}\label{22}
  a_h(u_h,u_h) \geq C \|u_h\|_{dG}^2,{\ }u_h\in \mathbb{B}_h(\cal{S}(\Omega))
 \end{align}
\end{lemma}
\begin{proof}
 By (\ref{8a}), we have that
 \begin{align}\label{24}
 \nonumber
  a_h(u_h,u_h)  =  \sum_{i=1}^N a_i(u_h,u_h)-s_i(u_h,u_h)+p_i(u_h,u_h) =\\
  \nonumber
                   \sum_{i=1}^N \alpha_i\|\nabla u_h\|^2_{L^2(\Omega_i)} - 
                   \sum_{F_{ij}\in \mathcal{F}} \frac{1}{2}\int_{F_{ij}} \{\alpha\nabla u_h\}\cdot \mathbf{n}_{F_{ij}} \llbracket u_h\rrbracket \,ds +\\
                   \sum_{F_{ij}\in \mathcal{F}} \mu\Big(\frac{\alpha^{(i)}}{h_i}+\frac{\alpha^{(j)}}{h_j}\Big)\|\llbracket u_h\rrbracket\|^2_{L^2(F_{ij})}.
 \end{align}
For the second term on the right hand side,  Lemma \ref{lemma4} and 
the trace inequality (\ref{18})
expressed on $F_{ij}\in \mathcal{F}$ yield the bound
\begin{multline}\label{25}
  -\sum_{F_{ij}\in \mathcal{F}} \frac{1}{2}\int_{F_{ij}} \{\alpha\nabla u_h\}\cdot \mathbf{n}_{F_{ij}} \llbracket u_h \rrbracket \,ds \geq \\
  \hspace*{-10mm}{
 - C_{1,\varepsilon}\sum_{i=1}^N \alpha_i\|\nabla u_h\|^2_{L^2(\Omega_i)} - 
    \sum_{F_{ij}\in \mathcal{F}}\frac{1}{C_{2,\varepsilon}}\Big(\frac{\alpha^{(i)}}{h_i}+\frac{\alpha^{(j)}}{h_j}\Big)\|\llbracket u_h \rrbracket\|^2_{L^2(F_{ij})}.
}
   \end{multline}
Inserting (\ref{25}) into (\ref{24}) and choosing $  C_{1,\varepsilon} <\frac{1}{2}$ and $\mu > \frac{2}{C_{2,\varepsilon}}$
we obtain (\ref{22}).\\

\end{proof}

\begin{lemma}\label{lemma6}(Boundedness) 
 There are $C_1,C_2>0$ independent of $h_i$ such that for all $ (u,\phi_h) \in W_h^{l,2}\times \mathbb{B}_h(\cal{S}(\Omega))$
 \begin{equation}\label{26}
  a_h(u,\phi_h)\leq C_1\Big( \|u\|_{dG}^2 +
  \sum_{F_{ij}\in \mathcal{F}} \alpha^{(i)}h_i\|\nabla u^{(i)}\|^2_{L^2(F_{ij})}\Big) 
                  +C_2\|\phi_h\|_{dG}^2.
 \end{equation}
\end{lemma}
\begin{proof}
 We have by (\ref{8a}) that
\begin{multline}\label{27_a}
 a_h(u,\phi_h) = \sum_{i=1}^N \int_{\Omega_i}\alpha\nabla u\nabla\phi_h\,dx +
         \sum_{F_{ij}\in \mathcal{F}}\frac{1}{2}\int_{F_{ij}}\{\alpha\nabla u\}\cdot \mathbf{n}_{F_{ij}} \llbracket \phi_h\rrbracket \,ds +\\
          \sum_{F_{ij}\in \mathcal{F}}\int_{F_{ij}} 
          \Big(\frac{\mu \alpha^{(j)}}{h_j}+\frac{\mu \alpha^{(i)}}{h_i} \Big) \llbracket  u\rrbracket \llbracket \phi_h \rrbracket\,ds
          =T_1+T_2+T_3.
          \end{multline}
Applying Cauchy-Schwartz inequality and consequently Young's inequality on every 
term in (\ref{27_a}) yield the bounds 
\begin{align*}
T_1&\leq  C_1\| u\|^2_{dG}+C_2\|\phi_h\|^2_{dG}. \\
& \text{For the term $T_2$, owing to the  Lemma \ref{lemma4}, we have}\\
T_2&\leq  \sum_{F_{ij}\in \mathcal{F}} \Big( C_1\alpha^{(i)}h_i\|\nabla u^{(i)}\|^2_{L^2(F_{ij})} +
           C_2  \Big(\frac{\mu \alpha^{(j)}}{h_j}+\frac{\mu \alpha^{(i)}}{h_i} \Big)\| \llbracket  \phi_h \rrbracket \|^2_{L^2(F_{ij})}\big)  \\
   &\leq         C_1\sum_{F_{ij}\in \mathcal{F}} \alpha^{(i)}h_i\|\nabla u^{(i)}\|^2_{L^2(F_{ij})} +
         C_2\|\phi_h\|^2_{dG}, \\ 
T_3& \leq \sum_{F_{ij}\in \mathcal{F}} \Big(\frac{\mu \alpha^{(j)}}{h_j}+\frac{\mu \alpha^{(i)}}{h_i} \Big)
                         \Big(C_1 \| \llbracket  u \rrbracket  \|^2_{L^2(F_{ij})} +
          C_2   \| \llbracket \phi_h \rrbracket\|^2_{L^2(F_{ij})} \Big)  
     \leq  C_1\|u\|^2_{dG}+C_2\|\phi_h\|^2_{dG}.   
\end{align*}
Substituting the bounds of $T_1,T_2,T_3$ into (\ref{27_a}), we can derive (\ref{26}).
\end{proof}

In Chp 12 in \cite{Shumaker_Bspline_book}, ${B}$-Spline quasi-intrpolants, say $\Pi_h$, 
are defined for $u\in W^{l,p}$ functions. 
Next, we  consider  the same quasi-interpolant   and  give an estimate on how well
$\Pi_h u$ approximates functions  $u\in W^{l,2}(\Omega_i)$ in $\|.\|_{dG}$-norm. 
\begin{lemma}\label{lemma8}
 Let $m,l\geq 2 $ be positive integers with $0\leq m \leq l \leq k+1$ and let $E=\mathbf{\Phi}_i(\hat{E}), 
 \hat{E}\in T^{(i)}_{h_i,\hat{D}}$.
 For $u\in W^{l,2}(\Omega_i)$ there exist a quasi-interpolant $\Pi_h u \in  \mathbb{B}_h^{(i)}$ and a 
 constant $C_i:=C_i(\max_{l_0<l}\|D^{l_0}\mathbf{\Phi}_i\|_{L^{\infty}(\Omega_i)},\|u\|_{W^{l,2}(\Omega_i)})$
 such that
 \begin{align}\label{29}
 \sum_{E\in T^{(i)}_{h_i,\Omega_i}} |u-\Pi_h u|^2_{W^{m,2}(E)} \leq 
 C_ih_i^{2(l-m)}  \|u\|^2_{W^{l,2}{(\Omega_i)}}.
 \end{align}
 Further, for any $F_{ij}\in \mathcal{F}$ the following  estimates are true
 \begin{subequations}\label{30}
 \begin{alignat}{1}\label{30a}
  h_i \alpha^{(i)} \|(\nabla u^{(i)}-\nabla \Pi_h u^{(i)})\cdot \mathbf{n}_{F_{ij}}\|^2_{L^{2}(F_{ij})} \leq  C_i  h_i^{2l-2},&  \\
  \label{30b}
  \big(\frac{\alpha^{(j)}}{h_j}+\frac{\alpha^{(i)}}{h_i}\big)\|u^{(i)}-\Pi_h u^{(i)} \|^2_{L^{2}(F_{ij})} \leq 
  C_i \big(\alpha^{(i)} h_i^{2l-2}+&\frac{\alpha^{(j)} h_i^{2l-1}}{h_j}\big),\\
  \label{30c}
  \|u-\Pi_h u\|_{dG}^2 \leq   \sum_{i=1}^N C_i \big(h_i^{2l-2}+  
  \sum_{F_{ij}\in \mathcal{F}}  \alpha^{(j)}\frac{ h_i}{h_j}h_i^{2l-2}\big).&
 \end{alignat}
 \end{subequations}
 
\end{lemma}
\begin{proof} 
 The proof of (\ref{29}) is included in  Lemma \ref{lemma5.3} (see below) if we set $p=2$, see also \cite{Bazilevs_IGA_ERR_ESti2006}.\\
  Applying the trace inequality (\ref{13b}) for $u:= u^{(i)}-\Pi_h u^{(i)}$ and consequently using 
  the approximation estimate (\ref{29}) the result (\ref{30a}) easily follows. \\ 
  To prove  (\ref{30b}), we apply again (\ref{13b}) and obtain  
  \begin{multline*}
   \big(\frac{\alpha^{(j)}}{h_j}+\frac{\alpha^{(i)}}{h_i}\big)\|u^{(i)}-\Pi_h u^{(i)} \|^2_{L^{2}(F_{ij})} \leq \\
 C_i\big(\frac{\alpha^{(j)}}{h_j}+\frac{\alpha^{(i)}}{h_i}\big)\big(\frac{1}{h_i}\|u^{(i)}-\Pi_h u^{(i)} \|^2_{L^{2}(\Omega_i)}  
 +h_i\|\nabla u^{(i)}-\nabla \Pi_h u^{(i)}\|^2_{L^2(\Omega_i)}\leq \\
 C_i\big(\frac{\alpha^{(j)}}{h_j}+\frac{\alpha^{(i)}}{h_i}\big)h_i^{2l-1} \leq 
 C_i \big(\alpha^{(i)} h_i^{2l-2}+ \frac{\alpha^{(j)}h_i^{2l-1}}{h_j} \big)
  \end{multline*}

  Recalling the approximation result (\ref{29})  and using  (\ref{30b}) we can deduce 
  estimate (\ref{30c}).\\

 \end{proof}

In order to proceed and to give an estimate for the error $\|u-u_h\|_{dG}$, we need to
show that the weak solution satisfies the form (\ref{8a}). 
\begin{lemma}(Consistency of the weak solution.)\label{lemma9}
 Under the Assumption \ref{Assumption1}, the weak solution $u$ of the 
 variational formulation (\ref{4}) satisfies the dG-IgA variational identity (\ref{8}), 
 that is for all $\phi_h \in \mathbb{B}_h(\cal{S}(\Omega))$, we have
  \begin{multline}\label{31}
 \sum_{i=1}^N\int_{\Omega_i}\alpha\nabla u \cdot \nabla \phi_h \,dx -
               \sum_{F_{ij}\in \mathcal{F}_I}\Big(   \int_{F_{ij}}\{\alpha\nabla u\}\cdot \mathbf{n}_{F_{ij}}\llbracket \phi_h \rrbracket \,ds +\\
            \big(\frac{\mu \alpha^{(i)}}{h_i} +  \frac{\mu \alpha^{(j)}}{h_j}\big)\int_{F_{ij}} \llbracket u \rrbracket 
            \llbracket \phi_h  \rrbracket \,ds \Big) +  \\
    \sum_{F_{i}\in \mathcal{F}_B}\Big(   \int_{F_{i}}\alpha^{(i)}\nabla u \cdot \mathbf{n}_{F_{i}} 
               \phi_h  \,ds  + \frac{\mu \alpha^{(i)}}{h_i}\int_{F_{i}}  u 
             \phi_h \,ds \Big) 
    = \\
 \sum_{i=1}^N  \int_{\Omega_i}f\phi_h \,dx +
  \sum_{F_{i}\in \mathcal{F}_B} \frac{\mu \alpha^{(i)}}{h_i}\int_{F_{i}}  u_D \phi_h \,ds. 
 \end{multline}

 \end{lemma}
 \begin{proof}
  We multiply (\ref{0}) by $\phi_h \in \mathbb{B}_h(\cal{S}(\Omega))$ and
  integrating by parts on each sub-domain $\Omega_i$ we get
  \begin{align*}
   \int_{\Omega_i}\alpha\nabla u \cdot \nabla \phi_h \,dx -
   \int_{\partial \Omega_i}\alpha \nabla u \cdot \mathbf{n}_{\partial \Omega_i} \phi_h \,ds= &
   \int_{\Omega_i}f\phi_h \,dx. 
  \end{align*}
   {Summing over all sub-domains}
   \begin{align}\label{31_a}
   \sum_{i=1}^N\int_{\Omega_i}\alpha\nabla u \cdot \nabla \phi_h \,dx -
               \sum_{F_{ij}\in \mathcal{F}}   \int_{F_{ij}} \llbracket \alpha\nabla u \phi_h \rrbracket \cdot \mathbf{n}_{F_{ij}}  \,ds=&
 \sum_{i=1}^N  \int_{\Omega_i}f\phi_h \,dx .
 \end{align}
 The regularity Assumption \ref{Assumption1}  implies that 
 $\llbracket \alpha\nabla u  \rrbracket \cdot \mathbf{n}_{F_{ij}}=0$.
Making use of the identity
    $$   \llbracket ab \rrbracket = a_1b_1-a_2b_2 = \{a\}\llbracket b \rrbracket + \llbracket a \rrbracket \{b\},$$
    the relation (\ref{31_a}) can be reformulated as
 \begin{align}\label{32}
 \sum_{i=1}^N\int_{\Omega_i}\alpha\nabla u \cdot \nabla \phi_h \,dx -
               \sum_{F_{ij}\in \mathcal{F}_I}\frac{1}{2}   \int_{F_{ij}}\{\alpha\nabla u\}\cdot \mathbf{n}_{F_{ij}}\llbracket 
               \phi_h \rrbracket \,ds + \\
    \nonumber
    \sum_{F_{i}\in \mathcal{F}_B}   \int_{F_{i}}\alpha\nabla u \cdot \mathbf{n}_{F_{i}} 
               \phi_h  \,ds 
   =
   \sum_{i=1}^N\int_{\Omega_i}f\phi_h \,dx.   
  \end{align}
    
The continuity of $u$ implies further  that  
 \begin{multline}\label{33}
 \sum_{F_{ij}\in \mathcal{F}_I} \Big(\frac{\mu \alpha^{(i)}}{h_i} + \frac{\mu \alpha^{(j)}}{h_j} \Big)\int_{F_{ij}} \llbracket u \rrbracket 
\llbracket \phi_h  \rrbracket \,ds +
\sum_{F_{i}\in \mathcal{F}_B} \frac{\mu \alpha^{(i)}}{h_i}\int_{F_{i}}  u 
             \phi_h \,ds  = \\        
 \sum_{F_{i}\in \mathcal{F}_B} \frac{\mu \alpha^{(i)}}{h_i}\int_{F_{i}}  u_D \phi_h \,ds. 
\end{multline}
Adding  (\ref{33}) and (\ref{32}) we obtain  (\ref{31}).\\
 \end{proof}

We can now give an error estimate  in $\|.\|_{dG}$-norm.
\begin{theorem}\label{lemma10}
 Let $u\in W^{l,2}_{\cal{S}}$ solves (\ref{4}) and let $u_h \in \mathbb{B}_h(\cal{S}(\Omega))$ solves 
 the discrete problem (\ref{8}).
 Then the error $u-u_h$ satisfies 
 \begin{equation}\label{34}
  \|u-u_h\|^2_{dG} < \sum_{i=1}^N  C_{i}\Big( h_i^{2l-2} + \sum_{F_{ij} \in \mathcal{F}} 
     \alpha^{(j)}\frac{ h_i}{h_j}h_i^{2l-2}\Big),\\
 \end{equation}
where   $C_{i}:=
        C(\max_{l_0<l}\|D^{l_0}\mathbf{\Phi}_i\|^l_{L^{\infty}(\Omega_i)},
        \|u\|_{W^{l,2}(\Omega_i)})$. 
\end{theorem}
\begin{proof}

 Let $\Pi_h u \in \mathbb{B}_h(\cal{S}(\Omega))$ as in Lemma \ref{lemma8}, by subtracting (\ref{31}) from (\ref{8a})  we get
 \begin{alignat*}{1}
   a_h(u_h,\phi_h)=a_h(u,\phi_h),     
 \end{alignat*}
and adding   $-a_h(\Pi_h u,\phi_h)$ on both sides 
 \begin{alignat}{1}\label{35}
   a_h(u_h-\Pi_h u,\phi_h)=a_h(u - \Pi_h u,\phi_h).     
 \end{alignat}
 Note that $u_h-\Pi_h u \in \mathbb{B}_h(\cal{S}(\Omega))$. Therefore  we may  set $\phi_h=u_h-\Pi_h u$ in  (\ref{35}),
 and consequently applying 
 Lemma \ref{lemma5} and Lemma \ref{lemma6} we find 
  \begin{alignat}{1}\label{36}
  \|u_h-\Pi_h u\|_{dG}^2 \leq C \Big(  \|u-\Pi_h u\|_{dG}^2 + \sum_{F_{ij}\in \mathcal{F}} \alpha^{(i)}h_i\|\nabla ( u^{(i)} -\Pi_h u^{(i)} )\|^2_{L^2(F_{ij})}\Big)   
  \end{alignat}
Using the  triangle inequality 
\begin{equation}\label{37}
 \|u- u_h\|_{dG}^2 \leq \|u_h-\Pi_h u\|_{dG}^2 + \|u-\Pi_h u\|_{dG}^2
\end{equation}
in (\ref{36}) and consequently applying the estimates of  (\ref{30}) 
we can obtain  (\ref{34}). \\

\end{proof}

\section{Low-Regularity solutions}
In this section, we investigate  the convergence  of the $u_h$
produced by the  dG-IgA method (\ref{8}), under the assumption that the
 weak solution $u$ of the  model problem (\ref{0})
 has less regularity, that is  
 $u\in W^{l,p}_{\cal{S}}:= W^{1,2}(\Omega) \cap W^{l,p}(\cal{S}(\Omega)),{\ }l\geq 2, p\in (\frac{2d}{d+2(l-1)},2]$. 
 Problems with low regularity solutions  can be found  in several  cases, as for example,
  when the domain has singular boundary points, points with changing boundary conditions,
   see e.g. \cite{Grisvard_EllpNoSmothDom}, \cite{MonigueDauge_Book}, 
   even in particular choices of the discontinuous    diffusion coefficient, \cite{Kellog_DiscDifCoef_1975}.
   We use the enlarged space $W_h^{l,p}= W^{l,p}_{\cal{S}}+\mathbb{B}_h(\cal{S}(\Omega))$ and
  will show that the dG-IgA method converges in optimal rate with respect to $\|.\|_{dG}$ norm
 defined in (\ref{20}). We develop our analysis inspired by the techniques used in 
 \cite{Riviere_DG_lowReg}, \cite{ERN_DGbook}. 
A basic tool that we will use is the Sobolev embeddings theorems, see \cite{Adams_Sobolevbook},\cite{Evans_PDEbook}.
Let $l=j+m\geq 2$, then for $j=0$ or $j=1$   it holds that
\begin{align}\label{38a}
\|u\|_{W^{j,2}(\Omega_i)}  \leq  C(l,p,2,\Omega_i) \|u\|_{W^{l,p}(\Omega_i)},{\ } \text{for}{\ }p> \frac{2d}{d+2m}.
\end{align}
We start by  proving  estimates on how well the quasi-interpolant $\Pi_hu$ defined in Lemma 
\ref{lemma8} approximates 
$u\in W^{l,p}(\Omega_i)$.
\begin{lemma}\label{lemma5.3}(Approximation estimates).
 Let $u\in W^{l,p}(\Omega_i)$ with 
 $l\geq 2, p\in (\max\{1,\frac{2d}{d+2(l-1)}\},2]$
 and let 
  $E=\mathbf{\Phi}_i(\hat{E}), \hat{E}\in T^{(i)}_{h_i,\hat{D}}$. Then for  $0 \leq m \leq l \leq k+1$, 
  there exist 
 constants $C_i:=C_i\big(\max_{l_0 \leq l}\|D^{l_0}\mathbf{\Phi}_i\|_{L^{\infty}(\Omega_i)}),
 \|u\|_{W^{l,p}(\Omega_i)}\big)$, such that
 \begin{align}\label{5.11}
 \sum_{E\in T^{(i)}_{h_i,\Omega_i}} |u-\Pi_h u|^p_{W^{m,p}(E)} \leq h_i^{p(l-m)} C_i. 
\end{align}
 Moreover,  we have the following estimates 
 \begin{subequations}\label{5.12}
 \begin{align}\label{5.12a}
 \bullet&{\ } h_i^{\beta} \|\nabla u^{(i)}-\nabla \Pi_h u^{(i)}\|^p_{L^{p}(F_{ij})} \leq  C_i C_{d,p} h_i^{p(l-1)-1+\beta},  \\
  \label{5.12b}
 \bullet&{\ } \big(\frac{\alpha^{(j)}}{h_j}+\frac{\alpha^{(i)}}{h_i}\big)
 \|\llbracket u-\Pi_h u\rrbracket \|^2_{L^{2}(F_{ij})} \leq \\  
  \nonumber
 &\quad C_i\alpha^{(j)}\frac{h_i}{h_j}\Big( h_i^{\delta (p,d)}\|u\|^p_{W^{l,p}(\Omega_i)}\Big)^{2} + 
  C_j\alpha^{(i)}\frac{h_j}{h_i}\Big( h_j^{\delta (p,d)}\|u\|_{W^{l,p}(\Omega_j)}\Big)^{2}+\\
  \nonumber
& \qquad  C_j\Big( h_j^{\delta (p,d)}\|u\|_{W^{l,p}(\Omega_j)}\Big)^{2} + 
  C_i\Big( h_i^{\delta (p,d)}\|u\|_{W^{l,p}(\Omega_i)}\Big)^{2}, \\
  \label{5.12c}
\bullet&{\ }  \|u-\Pi_h u\|^2_{dG} \leq   \sum_{i=1}^N C_i\Big( h_i^{\delta (p,d)}\|u\|_{W^{l,p}(\Omega_i)}\Big)^{2}  +\\
\nonumber
 &\qquad \qquad \sum_{F_{ij}\in\cal{F}}C_i\alpha^{(j)}\frac{h_i}{h_j}\Big( h_i^{\delta (p,d)}\|u\|_{W^{l,p}(\Omega_i)}\Big)^{2},
 \end{align}
 \end{subequations}
 where $\delta (p,d)= l+(\frac{d}{2}-\frac{d}{p}-1)$.
\end{lemma}
\begin{proof}
  We give the proof of (\ref{5.11}) based on the results of Chap 12 in \cite{Shumaker_Bspline_book}.
Given $f\in W^{l,p}(\hat{D})$, there exists a tensor-product polynomial $T^m f$ of order $m$,  
such that, for every
$\hat{E}\in T^{(i)}_{h_i,\hat{D}}$ the estimate 
\begin{align}\label{5.12cc}
 |f-T^mf|_{W^{m,p}(\hat{E})} \leq C_{d,l,m} h_i^{l-m}|f|_{W^{l,p}(D^{(i)}_{\hat{E}})},
\end{align}
holds, cf.  \cite{BrenerScotBook} and \cite{Shumaker_Bspline_book}.
Because of $m\leq k$ holds $\Pi_h(T^mf)=T^mf$  and   $\|\Pi_h f\|_{L^p(\hat{E})} 
\leq C \|f\|_{L^p(D^{(i)}_{\hat{E}})}$. Hence, we have that
\begin{multline}\label{5.12d}
 |u-\Pi_h u|_{W^{m,p}(\hat{E})} \leq |u-T^m u|_{W^{m,p}(\hat{E})} +  |\Pi_h u-T^m u|_{W^{m,p}(\hat{E})}  \\
                              \leq   |u-T^m u|_{W^{m,p}(\hat{E})} +  |\Pi_h ( u-T^m u)|_{W^{m,p}(\hat{E})}  \\
                            \quad  \leq  C_1 h_i^{l-m}|u|_{W^{l,p}(D^{(i)}_{\hat{E}})} + C_2 h_i^{-m+
                            \frac{d}{p}-\frac{d}{p}}|\Pi_h ( u-T^m u)|_{L^{p}(\hat{E})}  {\ }\text{(by (\ref{14}))} \\ 
                             \leq  C_1  h_i^{l-m}|u|_{W^{l,p}(D^{(i)}_{\hat{E}})} + C_2  h_i^{-m}| u-T^m u|_{L^{p}(\hat{E})}  {\ }\text{(by (\ref{5.12cc}))} \\
        \leq  C h_i^{l-m}| u|_{W^{l,p}(D^{(i)}_{\hat{E}})}.\qquad   
\end{multline}
Recalling (\ref{2_c}), the above inequality is expressed on every $E\in T^{(i)}_{h_i,\Omega_i}$. 
Then, taking  the $p-th$ power  and summing  over the elements 
we obtain the  estimate (\ref{5.11}).
\par
 We consider now the interface $F_{ij}=\partial \Omega_i\cap \Omega_j$.
 Applying  (\ref{13b}) and using the uniformity of the mesh we get
\begin{multline}\label{5.13}
h_i^{\beta} \|\nabla u^{(i)}-\nabla \Pi_h u^{(i)}\|^p_{L^{p}(F_{ij})} \leq 
 C_iC_{d,p} h_i^{\beta}(\frac{1}{h_i}\|\nabla u^{(i)}-\nabla \Pi_h u^{(i)}\|^p_{L^{p}(\Omega_i)}+ \\
      h_i^{p-1} \|\nabla^2 u^{(i)}-\nabla^2 \Pi_h u^{(i)}\|^p_{L^{p}(\Omega_i)})\leq^{\text{by (\ref{5.11})}} 
       C_iC_{d,p} h_i^{p(l-1)-1+\beta}.
\end{multline}
\par
To prove(\ref{5.12b}),  we again make use of the trace inequality (\ref{13b}) 
\begin{multline}\label{5.13a}
 \frac{\alpha^{(i)}}{h_i}\|u^{(i)}-\Pi_h u^{(i)} \|^2_{L^{2}(F_{ij})} \leq 
 C_iC_{d,p}\alpha^{(i)}\big(\frac{1}{h_i^{2}}\int_{\Omega_i}|u^{(i)}-\Pi_h u^{(i)}|^2 \,dx \\
 +\int_{\Omega_i}|\nabla (u^{(i)}-\Pi_h u^{(i)})|^2\,dx\big)= \\
 \hspace*{-10mm}{C_i  C_{d,p}\alpha^{(i)}\Big(\frac{1}{h_i^{2}} \sum_{E\in T^{(i)}_{h_i,\Omega_i}}
 \int_{E}|u^{(i)}-\Pi_h u^{(i)}|^2 \,dx 
 + \sum_{E\in T^{(i)}_{h_i,\Omega_i}}\int_{E}|\nabla (u^{(i)}-\Pi_h u^{(i)})|^2\,dx\Big).}
\end{multline}
The Sobolev embedding (\ref{38a}) gives
\begin{align}\label{5.15}
\|u\|_{L^2(D_p)} \leq  C(p,2,D_p)\big( \|u\|^p_{L^p(D_p)} + |u|^p_{W^{1,p}(D_p)}\big)^\frac{1}{p}.
\end{align}
 Using the  scaling arguments, see (\ref{9a}), and the bounds (\ref{2_c})  we can derive the 
 coresponding expression of  (\ref{5.15})  on every $E\in T^{(i)}_{h_i,\Omega_i}$, 
\begin{multline*}
h_i^{\frac{-d}{2}}\|u\|_{L^2(E)} \leq C_{i} h_i^{\frac{-d}{p}}\big( \|u\|^p_{{L^p(E)}} + h_i^p|u|^p_{W^{1,p}(E)}\big)^\frac{1}{p},  \\
\end{multline*}
\text{where a straight forward computation gives}
\begin{align}\label{5.16}
h_i^{-2}\|u\|^2_{L^2(E)} \leq C_{i} h_i^{2(\frac{d}{2}-\frac{d}{p}-1)}\big( \|u\|^p_{{L^p(E)}} + h_i^p|u|^p_{W^{1,p}(E)}\big)^\frac{2}{p}. 
\end{align}
Proceeding in the same manner, we can show 
\begin{multline}\label{5.17}
\hspace*{-2mm}{\|u\|^2_{W^{1,2}(E)} \leq C_{i} h_i^{2(\frac{d}{2}-\frac{d}{p}-1)}\big( \|u\|^p_{{L^p(E)}} +
                               h_i^p|u|^p_{W^{1,p}(E)} + h_i^{2p}|u|^p_{W^{2,p}(E)}\big)^\frac{2}{p}. }
\end{multline}
Setting in (\ref{5.16}) and (\ref{5.17}) $u:=u^{(i)}-\Pi_h u^{(i)}$ and 
applying the approximation estimate (\ref{5.11}), we obtain that
\begin{multline}\label{5.18}
 \sum_{E\in T^{(i)}_{h_i,\Omega_i}}\alpha^{(i)}\big(h_i^{-2}\|u^{(i)}-\Pi_h u^{(i)}\|^2_{L^2(E)} +\|u^{(i)}-\Pi_h u^{(i)}\|^2_{W^{1,2}(E)}\big) \\
\leq  \sum_{E\in T^{(i)}_{h_i,\Omega_i}}\big(\alpha^{(i)}C_i h_i^{l+(\frac{d}{2}-\frac{d}{p}-1)}\|u\|_{W^{l,p}(D^{(i)}_{E})}\big)^2 \leq {\ }
  \text{(note that $f(x)=(a^x+b^x)^{\frac{1}{x}} \downarrow$)}\\ 
 \alpha^{(i)}C_i\Big(\sum_{E\in T^{(i)}_{h_i,\Omega_i}}\big( h_i^{lp+p(\frac{d}{2}-\frac{d}{p}-1)}\|u\|^p_{W^{l,p}(D^{(i)}_{E})}\big)\Big)^{\frac{2}{p}} \leq
\alpha^{(i)}C_i\Big( h_i^{l+(\frac{d}{2}-\frac{d}{p}-1)}\|u\|_{W^{l,p}(\Omega_i)}\Big)^{2}.
\end{multline}
 Moreover, by (\ref{5.18}) we can deduce  that 
 \begin{multline}\label{5.18_a}
 \hspace*{-0.015mm}{
 \frac{\alpha^{(j)}h_i}{h_j}\frac{1}{h_i}\| u^{(i)}-\Pi_h u^{(i)} \|^2_{L^{2}(F_{ij})} \leq 
 C_i\frac{\alpha^{(j)}h_i}{h_j}\Big( h_i^{l+(\frac{d}{2}-\frac{d}{p}-1)}\|u\|_{W^{l,p}(\Omega_i)}\Big)^{2},
 }
\end{multline}
 similarly 
 \begin{multline}\label{5.18_b}
\hspace*{-0.010mm}{ \frac{\alpha^{(i)}h_j}{h_i}\frac{1}{h_j}\| u^{(j)}-\Pi_h u^{(j)} \|^2_{L^{2}(F_{ji})} \leq 
 C_i\frac{\alpha^{(i)}h_j}{h_i}\Big( h_j^{l+(\frac{d}{2}-\frac{d}{p}-1)}\|u\|_{W^{l,p}(\Omega_j)}\Big)^{2}.
 }
\end{multline}
Now, we return to the  left hand side of  (\ref{5.12b}) and use  (\ref{5.18}),(\ref{5.18_a}) and (\ref{5.18_b}),
 to obtain 
 \begin{multline}\label{5.18_c}
  \big(\frac{\alpha^{(j)}}{h_j}+\frac{\alpha^{(i)}}{h_i}\big)\|\llbracket u-\Pi_h u\rrbracket \|^2_{L^{2}(F_{ij})} \leq \\
  \frac{\alpha^{(j)}h_i}{h_j}\frac{1}{h_i}\| u^{(i)}-\Pi_h u^{(i)} \|^2_{L^{2}(F_{ij})} + 
  \frac{\alpha^{(i)}h_j}{h_i}\frac{1}{h_j}\| u^{(j)}-\Pi_h u^{(j)} \|^2_{L^{2}(F_{ji})}  \\
  +\frac{\alpha^{(j)}}{h_j}\| u^{(j)}-\Pi_h u^{(j)} \|^2_{L^{2}(F_{ji})}+ 
  \frac{\alpha^{(i)}}{h_i}\| u^{(i)}-\Pi_h u^{(i)} \|^2_{L^{2}(F_{ij})}  \\
\leq  C_i\frac{\alpha^{(j)}h_i}{h_j}\Big( h_i^{l+(\frac{d}{2}-\frac{d}{p}-1)}\|u\|_{W^{l,p}(\Omega_i)}\Big)^{2} + 
  C_j\frac{\alpha^{(i)}h_j}{h_i}\Big( h_j^{l+(\frac{d}{2}-\frac{d}{p}-1)}\|u\|_{W^{l,p}(\Omega_j)}\Big)^{2}\\
+  C_j\Big( h_j^{l+(\frac{d}{2}-\frac{d}{p}-1)}\|u\|_{W^{l,p}(\Omega_j)}\Big)^{2} + 
  C_i\Big( h_i^{l+(\frac{d}{2}-\frac{d}{p}-1)}\|u\|_{W^{l,p}(\Omega_i)}\Big)^{2}.
\end{multline} 
 For the proof (\ref{5.12c}), we recall the definition (\ref{20}) for $u-\Pi_h u$ and have
\begin{multline}\label{5.19}
\|u-\Pi_h u\|^2_{dG} = \sum_{i=1}^N\Big(\alpha^{(i)}\|\nabla (u^{(i)}-\Pi_h u^{(i)})\|^2_{L^2(\Omega_i)} \\
+               \sum_{F_{ij}\in\cal{F}} \Big( \frac{\mu\alpha^{(i)}}{h_i} + \frac{\mu\alpha^{(j)}}{h_j}\Big)
                \|\llbracket u-\Pi_h u\rrbracket\|^2_{L^2(F_{ij})} \Big). 
\end{multline}
Estimating  the first term on the right hand side in (\ref{5.19}) by (\ref{5.11}) and
the second term by (\ref{5.12b}), the approximation estimate (\ref{5.12c}) follows.\\
 
\end{proof}

We need further discrete coercivity, consistency and boundedness. 
The discrete coercivity (Lemma \ref{lemma5})  can also be applied here.
Using the same arguments as in Lemma \ref{lemma9}, we can prove the consistency 
for $u$. 
Due to assumed regularity of the solution,
the normal interface fluxe $(\alpha \nabla u)|_{\Omega_i}\cdot \mathbf{n}_{F_{ij}}$ belongs (in general)  to $L^p(F_{ij})$.
Thus, we need to prove the boundedness for $a_h(.,.)$ by estimating the flux terms (\ref{8d}) in different way
than this in Lemma \ref{lemma6}. 
 We work in a similar way as in \cite{Ern_DG_Hetrg_Diff} and  
 show  bounds for the interface fluxes in $\|.\|_{L^p}$ setting.
\begin{lemma}\label{lemma5.1}
 There is a constant $C:=C(p,2)$ such that the following inequality for 
 $(u,\phi_h)\in W_h^{l,p}\times \mathbb{B}_h(\cal{S}(\Omega))$ 
 holds true
 \begin{align}\label{5.2}
  \sum_{F_{ij}\in\mathcal{F}}\frac{1}{2}\int_{F_{ij}}\{\alpha\nabla u\}\cdot \mathbf{n}_{F_{ij}}\llbracket \phi_h \rrbracket \,ds \leq& \\ 
  \nonumber
   C\Big(\sum_{F_{ij}\in \mathcal{F}}
   \alpha^{(i)}h_i^{1+\gamma_{p,d}}\| \nabla u^{(i)}\|^p_{L^p(F_{ij})}+&
   \alpha^{(j)} h_j^{1+\gamma_{p,d}}\| \nabla u^{(j)}\|^p_{L^p(F_{ij})}\Big)^{\frac{1}{p}} \|\phi_h\|_{dG},\\
  \nonumber 
   \text{where}{\ }\gamma_{p,d} = \frac{1}{2}d(p-2).\hskip 2cm\qquad&
 \end{align}
\end{lemma}
\begin{proof}
 For the interface edge $e_{ij}\subset F_{ij}$   H\"older inequality yield
 \begin{multline}\label{5.3}
  \frac{1}{2}\int_{e_{ij}}\frac{1}{2}| \alpha^{(i)}\nabla u^{(i)}+ \alpha^{(j)}\nabla u^{(j)}|  |\llbracket \phi_h \rrbracket| \,ds \leq \\
  C\int_{e_{ij}} (\alpha^{(i)}h_i^{1+\gamma_{p,d}})^{\frac{1}{p}}|\nabla u^{(i)}| 
  \frac{\alpha^{(i)^{\frac{1}{q}}}}{h_i^{\frac{1+\gamma_{p,d}}{p}}}|\llbracket \phi_h \rrbracket| + 
  (\alpha^{(j)}h_j^{1+\gamma_{p,d}})^{\frac{1}{p}}|\nabla u^{(j)}| 
  \frac{\alpha^{(j)^{\frac{1}{q}}}}{h_j^{\frac{1+\gamma_{p,d}}{p}}}|\llbracket \phi_h \rrbracket|\,ds \\
   \leq C  (\alpha^{(i)}h_i^{1+\gamma_{p,d}})^{\frac{1}{p}}\|\nabla u^{(i)}\|_{L^p(e_{ij})} 
  \frac{\alpha^{(i)^{\frac{1}{q}}}}{h_i^{\frac{1+\gamma_{p,d}}{p}}}\|\llbracket \phi_h \rrbracket\|_{L^q(e_{ij})}  \\
  +C(\alpha^{(j)}h_j^{1+\gamma_{p,d}})^{\frac{1}{p}}\|\nabla u^{(j)}\|_{L^p(e_{ij})} 
  \frac{\alpha^{(j)^{\frac{1}{q}}}}{h_j^{\frac{1+\gamma_{p,d}}{p}}}\|\llbracket \phi_h \rrbracket\|_{L^q(e_{ij})}.
 \end{multline}
We  employ the inverse inequality (\ref{4.17}) with $p=q >2$, $q=2$ and  use the analytical form  
$\frac{1+\gamma_{p,d}}{p}=\frac{2+d(p-2)}{2p}$ to
express the jump terms in (\ref{5.3})  in the \textit{convenient $L^2$ form} as follows
\begin{align}\label{5.4}
\nonumber
\frac{\alpha^{(i)^{\frac{1}{q}}}}{h_i^{\frac{2+d(p-2)}{2p}}}\|\llbracket \phi_h \rrbracket\|_{L^q(e_{ij})} \leq &
C_{inv,p,2}\alpha^{(i)^{\frac{1}{q}}} h_{i}^{(d-1)(\frac{1}{q}-\frac{1}{2})-\frac{2+d(p-2)}{2p}} \|\llbracket \phi_h \rrbracket\|_{L^2(e_{ij})} \\
 &         \leq  C_{inv,p,2} \alpha^{(i)^{\frac{1}{q}}}h_i^{\frac{-1}{2}}\|\llbracket \phi_h \rrbracket\|_{L^2(e_{ij})}.
\end{align}
Inserting the result (\ref{5.4}) into (\ref{5.3}) and summing over all $e_{ij}\in F_{ij}$ we obtain for $q>2,$ 
\begin{multline}\label{5.5}
\frac{1}{2}\int_{F_{ij}}\{\alpha\nabla u\}\cdot \mathbf{n}_{F_{ij}}\llbracket \phi_h \rrbracket \,ds \leq C\sum_{e_{ij}\in F_{ij}}\int_{e_{ij}}| \alpha^{(i)}\nabla u^{(i)}+ \alpha^{(j)}\nabla u^{(j)}|  |\llbracket \phi_h \rrbracket| \,ds  \\
\leq C\Big( \sum_{e_{ij}\in F_{ij}}\alpha^{(i)} h_i^{1+\gamma_{p,d}}\|\nabla u^{(i)}\|^p_{L^p(e_{ij})}\Big)^{\frac{1}{p}} 
\Big(\sum_{e_{ij}\in F_{ij}}\alpha^{(i)}\big( \frac{1}{h_i^{\frac{1}{2}}}\|\llbracket \phi_h \rrbracket\|_{L^2(e_{ij})}\big)^q\Big)^{\frac{1}{q}} \\
+ C\Big( \sum_{e_{ij}\in F_{ij}} \alpha^{(j)}h_j^{1+\gamma_{p,d}}\|\nabla u^{(j)}\|^p_{L^p(e_{ij})}\Big)^{\frac{1}{p}} 
\Big(\sum_{e_{ij}\in F_{ij}}\alpha^{(j)}\big( \frac{1}{h_j^{\frac{1}{2}}}\|\llbracket \phi_h \rrbracket\|_{L^2(e_{ij})}\big)^q\Big)^{\frac{1}{q}}.
 \end{multline}
Now, using that the function 
$
f(x) = (\lambda\alpha^x+\lambda\beta^x)^{\frac{1}{x}}, {\ }\lambda>0,x>2
$ is decreasing, 
 we estimate  the ``q-power terms'' in the sum  of the right hand side in (\ref{5.5}) as follows
\begin{multline}\label{5.6}
\Big(\sum_{e_{ij}\in F_{ij}}\alpha^{(j)}\big( \frac{1}{h_j^{\frac{1}{2}}}\|\llbracket \phi_h \rrbracket\|_{L^2(e_{ij})}\big)^q\Big)^{\frac{1}{q}}
\leq  
\Big(\sum_{e_{ij}\in F_{ij}}\alpha^{(j)}\big( \frac{1}{h_j^{\frac{1}{2}}}\|\llbracket \phi_h \rrbracket\|_{L^2(e_{ij})}\big)^2\Big)^{\frac{1}{2}}  \\
\leq \Big( \big(\frac{\mu\alpha^{(i)}}{h_i}+\frac{\mu\alpha^{(j)}}{h_j}\big)\|\llbracket \phi_h \rrbracket\|^2_{L^2(F_{ij})}\Big)^{\frac{1}{2}}.
\end{multline}
Applying  (\ref{5.6}) into (\ref{5.5}) we get 
\begin{multline}\label{5.5a}
\frac{1}{2}\int_{F_{ij}}\{\alpha\nabla u\}\cdot \mathbf{n}_{F_{ij}}\llbracket \phi_h \rrbracket \,ds \leq \\
2C\Big( \alpha^{(i)}h_i^{1+\gamma_{p,d}}\|\nabla u^{(i)}\|^p_{L^p(F_{ij})} + \alpha^{(j)}h_j^{1+\gamma_{p,d}}\|\nabla u^{(j)}\|^p_{L^p(F_{ij})} \Big)^{\frac{1}{p}}\\
\Big( \big(\frac{\mu\alpha^{(i)}}{h_i}+\frac{\mu\alpha^{(j)}}{h_j}\big)\|\llbracket \phi_h \rrbracket\|^2_{L^2(F_{ij})}\Big)^{\frac{1}{2}}.
 \end{multline}
We sum  over all $F_{ij}\in \mathcal{F}$ in (\ref{5.5a}) and consequently we apply H\"older inequality
\begin{multline}\label{5.7}
\frac{1}{2}\sum_{F_{ij}\in\mathcal{F}}\int_{F_{ij}} \{ \alpha \nabla u \} \llbracket \phi_h \rrbracket \,ds \leq \\
2C\Big(\sum_{F_{ij}\in\mathcal{F}}\alpha^{(i)} h_i^{1+\gamma_{p,d}}\|\nabla u^{(i)}\|^p_{L^p(F_{ij})}+\alpha^{(j)} h_j^{1+\gamma_{p,d}}\|\nabla u^{(j)}\|^p_{L^p(F_{ij})}\Big)^{\frac{1}{p}} \\
\Big(\sum_{F_{ij}\in\mathcal{F}}\Big( \big(\frac{\mu\alpha^{(i)}}{h_i}+\frac{\mu\alpha^{(j)}}{h_j}\big)\|\llbracket \phi_h \rrbracket\|^2_{L^2(F_{ij})}\Big)^{\frac{q}{2}}\Big)^{\frac{1}{q}}.
\end{multline}

Following  in much the same arguments as  in proof of (\ref{5.6}),
we can bound the  second $\sum_{F_{ij}}$ in  (\ref{5.7}) as 
 \begin{multline}\label{5.7a}
\Big(\sum_{F_{ij}\in\mathcal{F}}\Big( \big(\frac{\mu\alpha^{(i)}}{h_i}+\frac{\mu\alpha^{(j)}}{h_j}\big)\|\llbracket \phi_h \rrbracket\|^2_{L^2(F_{ij})}\Big)^{\frac{q}{2}}\Big)^{\frac{1}{q}} \leq \\
\Big(\sum_{F_{ij}\in\mathcal{F}} \big(\frac{\mu\alpha^{(i)}}{h_i}+\frac{\mu\alpha^{(j)}}{h_j}\big)\|\llbracket \phi_h \rrbracket\|^2_{L^2(F_{ij})}\Big)^{\frac{1}{2}} \leq 
 \|\phi_h\|_{dG}.
\end{multline}
 Using  (\ref{5.7a}) in (\ref{5.7}), we can easily obtain (\ref{5.2}).\\
 
\end{proof}

\begin{lemma}\label{lemma5.2}(boundedness)
 There is  a $C:=C_{p,2}$ independent of $h_i$ such that $\forall (u,\phi_h)\in W^{l,p}_h\times \mathbb{B}_h(\cal{S}(\Omega))$
 \begin{align}\label{5.8}
  a_h(u,\phi_h)\leq C (\|u\|_{dG}^p + \sum_{F_{ij}\in \mathcal{F}}
   h_i^{1+\gamma_{p,d}}\alpha^{(i)}\| \nabla u^{(i)}\|^p_{L^p(F_{ij})}+ \\
   \nonumber
   h_j^{1+\gamma_{p,d}}\alpha^{(j)}\| \nabla u^{(j)}\|^p_{L^p(F_{ij})}\Big)^{\frac{1}{p}} \|\phi_h\|_{dG},
 \end{align}
\end{lemma}
\begin{proof}
We estimate the terms of  $a_h(u,\phi_h)$ in (\ref{8b}) separately.
 Applying Cauchy-Schwartz for the terms (\ref{8c}) and (\ref{8e}) we have
 \begin{subequations}\label{5.9}
 \begin{align}\label{5.9a}
     \sum_{i=1}^N a_i(u,\phi_h) \leq C\|u\|_{dG} \|\phi_h\|_{dG}& \qquad \\
     \label{5.9b}
     \sum_{i=1}^N p_i(u,\phi_h) \leq C\|u\|_{dG} \|\phi_h\|_{dG}.& \qquad
 \end{align}
 \end{subequations}
 For the term (\ref{8d}) we use Lemma \ref{lemma5.1} 
 \begin{multline}\label{5.10}
 \sum_{i=1}^N s_i(u,\phi_h) \leq  
 C\Big(\sum_{F_{ij}\in \mathcal{F}}
   \alpha^{(i)}h_i^{1+\gamma_{p,d}}\| \nabla u^{(i)}\|^p_{L^p(F_{ij})}+ \\
   \alpha^{(j)}h_j^{1+\gamma_{p,d}}\| \nabla u^{(j)}\|^p_{L^p(F_{ij})}\Big)^{\frac{1}{p}} \|\phi_h\|_{dG},
 \end{multline}
Combining (\ref{5.9}) with (\ref{5.10})  we can derive (\ref{5.8}).\\
 
\end{proof}

Next,  we prove  the main convergence result of this section. 
\begin{theorem}
 Let $u\in W^{l,p}_{\cal{S}},l\geq 2,{\ } p\in (\max\{1,\frac{2d}{d+2(l-1)}\},2]$ be the solution of (\ref{4a}).
 Let $u_h\in \mathbb{B}_h(\cal{S}(\Omega))$ be the dG-IgA solution of (\ref{8a}) and $\Pi_h u\in \mathbb{B}_h(\cal{S}(\Omega))$ is
 the interpolant of Lemma \ref{lemma5.3}. Then there are \\
 $C_i:=C_i(\max_{l_0 \leq l}\|D^{l_0}\mathbf{\Phi}_i\|_{L^{\infty}(E)}),
 \|u\|_{W^{l,p}(\Omega_i)}\big)$, such that
 \begin{align}\label{5.24}
   \|u-u_h \|_{dG} \leq &  \sum_{i=1}^N\Big( C_i\Big( h_i^{\delta (p,d)} +
 \sum_{F_{ij}\in \mathcal{F}}\alpha^{(j)}\frac{h_i}{h_j}  h_i^{\delta (p,d)}\Big)\|u\|_{W^{l,p}(\Omega_i)}\Big),
 \end{align}
 where $\delta (p,d)= l+(\frac{d}{2}-\frac{d}{p}-1)$.
 \end{theorem}
 \begin{proof}
  Since $( u_h-\Pi_h u) \in \mathbb{B}_h(\cal{S}(\Omega))$ by the discrete coercivity (\ref{22}) we have
  \begin{align}\label{5.25}
   \|u_h-\Pi_h u\|_{dG}^2 \leq a_h(u_h-\Pi_h u,u_h-\Pi_h u).
  \end{align}
By orthogonality we have
\begin{multline*}
  \|u_h-\Pi_h u\|_{dG}^2 \leq  a_h(u_h-\Pi_h u,u_h-\Pi_h u)= \\
   a_h\big( (u_h -u) +(u-\Pi_h u),u_h-\Pi_h u\big)= a_h\big(u-\Pi_h u,u_h-\Pi_h u)  \\
\leq C\Big(\|u-\Pi_h u\|_{dG} + \Big(\sum_{F_{ij}\in \mathcal{F}}
   h_i^{1+\gamma_{p,d}}\alpha^{(i)}\| \nabla u^{(i)}-\Pi_h u^{(i)}\|^p_{L^p(F_{ij})}\\
 ^+  h_j^{1+\gamma_{p,d}}\alpha^{(j)}\| \nabla u^{(j)}-\Pi_h  u^{(j)}\|^p_{L^p(F_{ij})}\Big)^{\frac{1}{p}}\Big) \|u_h-\Pi_h u\|_{dG},
 \end{multline*}
 where immediately we get 
 \begin{multline}\label{5.26}
  \|u_h-\Pi_h u\|_{dG} \leq \|u-\Pi_h u\|_{dG} + \Big(\sum_{F_{ij}\in \mathcal{F}}
   h_i^{1+\gamma_{p,d}}\alpha^{(i)}\| \nabla u^{(i)}-\Pi_h u^{(i)}\|^p_{L^p(F_{ij})}\\
 \quad + h_j^{1+\gamma_{p,d}}\alpha^{(j)}\| \nabla u^{(j)}-\Pi_h  u^{(j)}\|^p_{L^p(F_{ij})}\Big)^{\frac{1}{p}}.
  \end{multline}
  Now, using triangle inequality, the approximation estimates (\ref{5.12}) and the bound (\ref{5.2}) 
  in   (\ref{5.26}), we obtain
  \begin{multline}\label{5.30}
 \|u_h- u\|_{dG} \leq  \|u_h-\Pi_h u\|_{dG} + \|u-\Pi_h u\|_{dG} \leq \\
 \sum_{i=1}^N C_i h_i^{\delta (p,d)}\|u\|_{W^{l,p}(\Omega_i)} + 
 \sum_{F_{ij}\in \mathcal{F}}C_i\frac{\alpha^{(j)}h_i}{h_j} h_i^{\delta (p,d)}\|u\|_{W^{l,p}(\Omega_i)}, 
  \end{multline}
  which is the required error estimate   (\ref{5.24}).\\
   
 \end{proof}

\section{Numerical examples}
In this section, we present a series of numerical 
examples to validate numerically the theoretical results, which were previously  shown. 
We restrict ourselves for a model problem in $\Omega=(\frac{-1}{2},\frac{1}{2})^{d=3}$, 
with $\Gamma_D=\partial \Omega$. 
The domain $\Omega$ is subdivided in  four equal sub-domains $\Omega_i, i=1,...,4$,
 where for simplicity every $\Omega_i$ is initially partitioned into a mesh 
 $T^{(i)}_{h_i,\Omega_i}$ with  $h:=h_i=h_j, i\neq j, i,j=1,...,4$.
 Successive uniform refinements are performed 
 on every $T^{(i)}_{h_i,\Omega_i}$ in order to compute numerically
 the convergence rates. 
We set the diffusion coefficient equal to one.

\par
All the numerical tests have been performed in G+SMO\footnote{G+SMO: http://www.gs.jku.at/gs-gismo.shtml}, 
which is a generic object oriented C++ library for IgA computations.
For the reasons mentioned in Remark~\ref{remark_3.1},
the practical implementation in G+SMO is based on SIP dG-IgA.
In the first test, 
the data $u_D$ and $f$ in (\ref{0}) are determined so that the exact solution  is given by
$u(x)=\sin(2.5\pi x)\sin(2.5\pi y)\sin(2.5\pi z)$ (smooth test case).
The first two columns of Table \ref{table_1} display the convergence rates. 
As it was expected, the convergence rates are optimal.  
In the second case, the exact solution is 
$u(x)=|x|^{\lambda}$. The parameter $\lambda$ is chosen such that $u\in W^{l,p=1.4}(\Omega)$.
In the  last columns of Table \ref{table_1},
we display the convergence rates for degree $k=2,{\ }k=3$ and $l=2, {\ }l=3$.
We observe that, for each of the two different tests, the error in the 
dG-norm behaves according to the main error estimate given by (\ref{5.24}).

\begin{table}[hutb]
\centering %
\begin{tabular}{|c| |c|c|c|c|c|c|} 
\hline 
          &\multicolumn{2}{|c|}{highly smooth }
          &\multicolumn{2}{|c|}{$k =2$ } &\multicolumn{2}{|c|}{$k =3$ }\\ [0.5ex] \hline 
$\frac{h}{2^s}$    &$k=2$& $k=3$&  $l=2$ & $l=3$        &  $l=2$ & $l=3$  \\ \hline\hline
-           &\multicolumn{6}{|c|} {Convergence rates } \\ [0.5ex] \hline
 $s=0$  &-    & -     & -     &-  & -      & -  \\ \hline
 $s=1$  &0.15 &2.91   & 0.62  &0.76   & 0.24   &1.64    \\ \hline
 $s=2$  &2.34 &2.42   & 0.29  &1.10   & 0.28   &0.89    \\ \hline
 $s=3$  &2.08 &3.14   & 0.35  &1.32   & 0.47   &1.25  \\ \hline  
 $s=4$  &2.02 &3.04   & 0.35  &1.36   & 0.36   &1.37     \\ \hline
\end{tabular}
\caption{The numerical convergence rates of the dG-IgA method.}
\label{table_1}
\end{table}
\begin{rem}
 In a forthcoming paper, we will present graded mesh techniques in dG-IgA methods for
          treating problems with low regularity solutions. We will show,  how to construct
          graded refined mesh in the vicinity of the singular points of $u$, in order 
          to get the optimal approximation order as in the case of having smooth $u$.
\end{rem}
\section{Conclusions}
In this paper, we presented theoretical  error estimates 
of the dG-IgA  method applied to a model 
elliptic problem with discontinuous coefficients. The problem was discretized  according to
IgA methodology using discontinuous $\mathbb{B}$-Spline spaces.
Due to global discontinuity of the approximate solution on the sub-domain interfaces,
dG discretizations techniques were utilized. 
In the first part, we assumed higher   regularity
for the exact solution, that is $u\in W^{l\geq 2,2}$,  and we showed 
optimal error estimates with respect to $\|.\|_{dG}$. 
 In the second part, we assumed low regularity for the exact solution, that is 
 $u\in W^{l\geq 2,p\in (\frac{2d}{d+2(l-1)},2)}$, and applying the Sobolev embedding theorem we  
 proved optimal convergence rates with respect to $\|.\|_{dG}$. The theoretical error estimates
 were validated by numerical tests. 
The results can obviously be carried over to diffusion problems on open and closed
surfaces as studied in \cite{LangerMoore:2014a}, and to more general second-order 
boundary value problems like linear elasticity problems 
as studied in \cite{ApostolatosSchmidtWuencherBletzinger:2014a}.

\section{Acknowledgments}
The authors thank  Angelos Mantzaflaris for his help on performing the numerical tests.
This work was supported by Austrian Science Fund (FWF) under the grant NFN S117-03.

\bibliographystyle{plain}
\bibliography{DG_IGA_lnrEllp}

\begin{thebibliography}{10}

\bibitem{Adams_Sobolevbook}
R.~A. Adams and J.~J.~F. Fournier.
\newblock {\em Sobolev {S}paces}, volume 140 of {\em Pure and {A}pplied
  {M}athematics}.
\newblock {ACADEMIC PRESS}-imprint {E}lsevier {S}cience, second {E}dition
  edition, 2003.

\bibitem{ApostolatosSchmidtWuencherBletzinger:2014a}
A.~Apostolatos, R.~Schmidt, R.~W\"uchner, and K.-U. Bletzinger.
\newblock A {N}itsche-type formulation and comparison of the most common domain
  decomposition methods in isogeometric analysis.
\newblock {\em Int. J. Numer. Meth. Engng.}, 97(7):1099--1142, 2014.

\bibitem{CockMurEllipDG}
D.~N. Arnold, F.~Brezzi, B.~Cockburn, and D.~L. Marini.
\newblock Unified analysis of discontinuous {G}alerkin methods for elliptic
  problem.
\newblock {\em SIAM J. Numer. Anal.}, 39(5):1749--1779, 2002.

\bibitem{Bazilevs_IGA_ERR_ESti2006}
Y.~Bazilevs, L.~Beir{a}o da~Veiga, J.~A. Cottrell, T.~J.~R. Hughes, and
  G.~Sangalli.
\newblock Isogeometric analysis: approximation, stability and error estimates
  for $h$-refined meshes.
\newblock {\em Math. Mod. Meth. Appl. Sci.}, 16(7):1031--1090, 2006.

\bibitem{BrenerScotBook}
S.~C. Brener and L.~R. Scott.
\newblock {\em The mathematical Theory of Finite Element Methods}, volume~15 of
  {\em Texts in {A}pplied {M}athematics}.
\newblock Springer-Verlag New York, third {E}dition edition, 2008.

\bibitem{Brunero:2012a}
F.~Brunero.
\newblock Discontinuous {G}alerkin methods for isogeometric analysis.
\newblock Master's thesis, Universit\`{a} degli Studi di Milano, 2012.

\bibitem{Ciarlet_FEbook}
P.~G. Ciarlet.
\newblock {\em The {F}inite {E}lement {M}ethod for {E}lliptic {P}roblems}.
\newblock Studies in {M}athematics and its {A}pplications. North {H}olland
  {P}ublishing {C}ompany, 1978.

\bibitem{Hughes_IGAbook}
J.~A. Cotrell, T.~J.~R. Hughes, and Y.~Bazilevs.
\newblock {\em Isogeometric {A}nalysis, Toward {I}ntegration of {CAD} and
  {FEA}}.
\newblock John Wiley and Sons, 2009.

\bibitem{Beirao_Buffa_NumMathIGAErr2011}
L.~Beir{a}o da~Veiga, A.~Buffa, J.~Rivas, and G.~Sangalli.
\newblock Some estimates for $hpk-$ refinement in {I}sogeometric {A}nalysis.
\newblock {\em Numer. Math.}, 118(7):271--305, 2011.

\bibitem{BeiraoChoPavarinoScacchi:2013a}
L.~Beir{a}o da~Veiga, D.~Cho, L.~Pavarino, and S.~Scacchi.
\newblock {BDDC} preconditioners for isogeometric analysis.
\newblock {\em Math. Models Methods Appl. Sci.}, 23(6):1099--1142, 2013.

\bibitem{MonigueDauge_Book}
M.~Dauge.
\newblock {\em Elliptic {B}oundary {V}alue {P}roblems on {C}orner {D}omains},
  volume 1341 of {\em Lecture {N}otes in {M}athematics}.
\newblock Springer Verlag, 1988.

\bibitem{Maximiliam_DG_DD}
M.~Dryja.
\newblock On discontinuous {G}alerkin methods for elliptic problems with
  discontinuous coefficients.
\newblock {\em Comput. Meth. Appl. Math.}, 3(1):76--85, 2003.

\bibitem{Maximiliam_DG_BDDC}
M.~Dryja, J.~Galvis, and M.~Sarkis.
\newblock {BDDC} methods for discontinuous {G}alerkin discretization of
  elliptic problems.
\newblock {\em J. Complexity}, 23:715--739, 2007.

\bibitem{EvansHughes:2013a}
J.~A. Evans and T.~J.~R. Hughes.
\newblock Isogeometric {D}ivergence-conforming {B}-splines for the
  {D}arcy-{S}tokes-{B}rinkman equations.
\newblock {\em Math. Models Methods Appl. Sci.}, 23(4):671--741, 2013.

\bibitem{Evans_PDEbook}
L.~C. Evans.
\newblock {\em Partial {D}ifferential {E}questions}, volume~19 of {\em Graduate
  {S}tudies in {M}athematics}.
\newblock American {M}athematical {S}ociety, 1st {E}dition edition, 1998.

\bibitem{Karakashian_SCHWARTZ_DG_Elliptic}
X.~Feng and O.~A. Karakashian.
\newblock Two-level additive {S}chwarz methods for a discontinuous {G}alerkin
  approximation of second order elliptic problems.
\newblock {\em SIAM J. Numer. Anal}, 39(4):1343--1365, 2001.

\bibitem{Grisvard_EllpNoSmothDom}
P.~Grisvard.
\newblock {\em Elliptic {P}roblems in {N}onsmooth {D}omains}.
\newblock Number~69 in Classics in Applied Mathematics. SIAM, 2011.

\bibitem{HEINRICH_2003_IMA}
B.~Heinrich and S.~Nicaise.
\newblock The {N}itsche mortar finite-element method for transmission problems
  with singularities.
\newblock {\em IMA J. Numer. Anal}, 23:331--358, 2003.

\bibitem{Hughes_FEbook}
T.~J.~R. Hughes.
\newblock {\em The {F}inite {E}lement {M}ethod: Linear {S}tatic and {D}ynamic
  {F}inite {E}lement {A}nalysis}.
\newblock Dover Publications, 2000.

\bibitem{HUGHE_IGA_CMAME_2005}
T.~J.~R. Hughes, J.~A. Cottrell, and Y.~Bazilevs.
\newblock Isogeometric analysis : {CAD}, finite elements, {NURBS}, exact
  geometry and mesh refinement.
\newblock {\em Comput. Methods Appl. Mech. Engrg.}, 194:4135--4195, 2005.

\bibitem{Kellog_DiscDifCoef_1975}
R.~B. Kellogg.
\newblock On the {P}oisson equation with intersecting interfaces.
\newblock {\em Appl. Anal.}, 4:101--129, 1975.

\bibitem{KirbyKarniadakis2005}
R.~M. Kirby and Em.~G. Karniadakis.
\newblock Selecting the numerical flux in discontinuous {G}alerkin methods for
  diffusion problems.
\newblock {\em J. Sci. Comput.}, 22 and 23:385--411, 2005.

\bibitem{KleissPechsteinJuttlerTomar:2012a}
S.~K. Kleiss, C.~Pechstein, B.~J\"uttler, and S.~Tomar.
\newblock {IETI} -- {I}sogeometric {T}earing and {I}nterconnecting.
\newblock {\em Comput. Methods Appl. Mech. Engrg.}, 247 - 248(0):201--215,
  2012.

\bibitem{Knees:2004a}
D.~Knees.
\newblock On the regularity of weak solutions of quasi-linear elliptic
  transmission problems on polyhedral domains.
\newblock {\em Z. Anal. Anwendungen}, 23(3):509--546, 2004.

\bibitem{LangerMoore:2014a}
U.~Langer and S.~E. Moore.
\newblock Discontinuous {G}alerkin isogeometric analysis of elliptic {PDE}s on
  surfaces.
\newblock NFN Technical Report~12, Johannes Kepler University Linz, NFN
  Geometry and Simulation, Linz, 2014.
\newblock http://arxiv.org/abs/1402.1185, accepted for publication in the
  proceedings of the 22nd International Domain Decomposition Conference.

\bibitem{BQ_Li_DGbook}
B.~Q. Li.
\newblock {\em Discontinuous finite element in fluid dynamics and heat
  transfer}.
\newblock Computational Fluid and Solid Mechanics. London Springer Verlag,
  2006.

\bibitem{ERN_DGbook}
D.~A.~Di Pietro and A.~Ern.
\newblock {\em Mathematical {A}spects of {D}iscontinuous {G}alerkin {M}ethods
  ({M}athématiques et {A}pplications)}, volume~69 of {\em {M}athématiques et
  {A}pplications}.
\newblock Springer-Verlag, 2010.

\bibitem{Ern_DG_Hetrg_Diff}
D.~A.~Di Pietro and A.~Ern.
\newblock Analysis of a {D}iscontinuous {G}alerkin {M}ethod for {H}eterogeneous
  {D}iffusion {P}roblems with {L}ow-{R}egularity {S}olutions.
\newblock {\em Numer. Methods Partial Diff. Equations}, 28(4):1161--1177, 2012.

\bibitem{Piegl_NURBS_book}
L.~A. Pigl and W.~Tiller.
\newblock {\em The {NURBS} book}.
\newblock Springer Verlag, 1997.

\bibitem{Rivierebook}
B.~Riviere.
\newblock {\em Discontinuous {G}alerkin methods for Solving Elliptic and
  Parabolic Equations}.
\newblock SIAM, Society for Industrial and Applied Mathematics Philadelphia,
  2008.

\bibitem{RiviereWheelerGirault2001}
B.~Riviere, M.~Wheeler, and V.~Girault.
\newblock A priori error estimates for finite element methods based on
  discontinuous approximation spaces for elliptic problems.
\newblock {\em SIAM J. Numer. Anal}, 39(2):902--931, 2001.

\bibitem{Shumaker_Bspline_book}
L.~L. Schumaker.
\newblock {\em Spline {F}unctions: {B}asic {T}heory}.
\newblock Cambridge, University Press, third edition, 2007.

\bibitem{Turek_INS_book}
S.~Turek.
\newblock {\em Efficient {S}olvers for {I}ncompressible {F}low {P}roblems: {A}n
  {A}lgorithmic and {C}omputational {A}pproach}, volume~6 of {\em Lecture Notes
  in Computational Science and Engineering}.
\newblock Springer, 2013.

\bibitem{Riviere_DG_lowReg}
T.~P. Wihler and B.~Riviere.
\newblock Discontinuous {G}alerkin {M}ethods for {S}econd-{O}rder {E}lliptic
  {PDE} with {L}ow-{R}egularity {S}olutions.
\newblock {\em J. Sci. Comput.}, 46(2):151--165, 2011.

\bibitem{Zlamal:1973a}
M.~Zlam\'{a}l.
\newblock The finite element method in domains with curved boundaries.
\newblock {\em Int. J. Numer. Meth. Engng.}, 5(3):367--373, 1973.

\end{thebibliography}

\end{document}